\documentclass{amsart}[12pt]
\usepackage{graphics}
\usepackage{amsmath,amsfonts,amssymb,amsthm,mathtools}

\numberwithin{equation}{section}

\usepackage[enableskew,vcentermath]{youngtab}
\usepackage{ytableau,mathdots,yhmath}
\usepackage[all]{xy}
\usepackage{enumerate}
\usepackage[colorinlistoftodos]{todonotes}
\usepackage{url}
\usepackage{hyperref}
\usepackage{comment}

\RequirePackage{cleveref}
\usepackage{hypcap}
\hypersetup{colorlinks=true, citecolor=darkblue, linkcolor=darkblue}
\definecolor{darkblue}{rgb}{0.0,0,0.7}

\usepackage{graphicx} % Required for inserting images
\usepackage{amsmath}
\usepackage{amsthm}
\usepackage{amssymb}
\usepackage{ytableau}
\usepackage{bm}
\usepackage{url}
\usepackage{geometry} 
\usepackage{setspace, xcolor,mathrsfs}
\newtheorem{theorem}{Theorem}[section]
\newtheorem{proposition}[theorem]{Proposition}

\newtheorem{lemma}[theorem]{Lemma}
\newtheorem{corollary}[theorem]{Corollary}
\newtheorem{conjecture}[theorem]{Conjecture}
\newtheorem*{conjecture*}{Conjecture}
\newtheorem*{proposition*}{Proposition}
\theoremstyle{definition}
\newtheorem{definition}[theorem]{Definition}
\newtheorem{example}[theorem]{Example}
\newtheorem{remark}[theorem]{Remark}

\newcommand{\Josh}[1]{\todo[size=\tiny,inline,color=green!30]{#1
      \\ \hfill --- A.}}

\newcommand{\Jianping}[1]{\todo[size=\tiny,inline,color=magenta!30]{#1
      \\ \hfill --- P.}}

\newcommand{\UU}{\mathcal{U}}
\newcommand{\HH}{\mathcal{H}}

\newcommand{\GQ}{\mathrm{GQ}}

\newcommand{\st}{\mathsf{st}}
\newcommand{\Peak}{\mathsf{Peak}}

\newcommand{\SYT}{\mathsf{SYT}}

\newcommand{\ShSSYT}{\mathsf{ShSSYT}}
\newcommand{\ShSVT}{\mathsf{ShSVT}}

\newcommand{\res}{\mathsf{res}}

\definecolor{darkred}{rgb}{0.7,0,0} 
\newcommand{\defn}[1]{{\color{darkred}\emph{#1}}} % emphasis of a definition

\newcommand{\N}{\mathbb{N}}
\newcommand{\Z}{\mathbb{Z}}

\newcommand{\fkS}{\mathfrak{S}}

\newcommand{\bfa}{\textbf{a}}
\newcommand{\bfb}{\textbf{b}}
\newcommand{\bfi}{\textbf{i}}
\newcommand{\bfx}{\textbf{x}}

\newcommand{\SDT}{\mathrm{SDT}}
\newcommand{\ShSet}{\mathrm{ShSet}}

\renewcommand{\to}[1]{\overset{#1}{\longrightarrow}}

\title[Kra\'skiewicz--Hecke insertion]{Type C $K$--stanley symmetric functions and Kra\'skiewicz--Hecke insertion}

\author[Arroyo]{Joshua Arroyo}
\address{Department of Mathematics,  University of Florida, Gainesville, FL 32611}
\email{joshuaarroyo@ufl.edu}

\author[Hamaker]{Zachary Hamaker}
\address{Department of Mathematics,  University of Florida, Gainesville, FL 32611}
\email{zhamaker@ufl.edu}

\author[Hawkes]{Graham Hawkes}
\address{Industry}
\email{ghawkes1217@gmail.com}
	
\author[Pan]{Jianping Pan}
\address{School of Mathematical and Statistical Sciences, Arizona State University, AZ 85287}
\email{jianping.pan@asu.edu}

\thanks{\today}

\begin{document}

\begin{abstract}
We study Type C $K$--Stanley symmetric functions, which are $K$--theoretic extensions of the Type C Stanley symmetric functions.
They are indexed by signed permutations and can be used to enumerate reduced words via their expansion into Schur $Q$-functions, which are indexed by strict partitions.
A combinatorial description of the Schur $Q$- coefficients is given by Kra\'skiewicz insertion.
Similarly, their $K$--Stanley analogues are conjectured to expand positively into $GQ$'s, which are $K$--theory representatives for the Lagrangian Grassmannian introduced by Ikeda and Naruse also indexed by strict partitions.
We introduce a $K$--theoretic analogue of Kra\'skiewicz insertion, which can be used to enumerate 0--Hecke expressions for signed permutations and gives a conjectural combinatorial rule for computing this $GQ$ expansion.

We show the Type C $K$--Stanleys for certain fully commutative signed permutations are skew $GQ$'s.
Combined with a Pfaffian formula of Anderson's, this allows us to prove Lewis and Marberg's conjecture that $GQ$'s of (skew) rectangle  shape are $GQ$'s of trapezoid shape.
Combined with our previous conjecture, this also gives an explicit combinatorial description of the skew $GQ$ expansion into $GQ$'s.
As a consequence, we obtain a conjecture for the product of two $GQ$ functions where one has trapezoid shape.

\end{abstract}

\maketitle
%\tableofcontents

\section{Introduction}

Although this paper is combinatorial in methods and results, our underlying objective is to understand geometric properties of the isotropic or Lagrangian Grassmannian.
Beginning in the 19th century with work of Schubert, mathematicians have investigated enumerative properties of the intersections of curves and hypersurfaces.
The modern approach to calculating these quantities is to interpret them by computing the cup product for the cohomology ring of an appropriate projective variety.
In this guise, such computations are fundamentally combinatorial in nature.
For example, Schubert's computations are encoded in the cohomology of the Grassmannian $\mathrm{Gr}(k,n)$, the space of $k$--planes in $\mathbb{C}^n$, whose cup product is computed by the Littlewood--Richardson rule for multiplying Schur functions (see e.g.~\cite{manivel2001symmetric} for a textbook treatment).
Similarly, the cup product in cohomology for the orthogonal and Lagrangian Grassmannians is computed by multiplying Schur $P$- and Schur $Q$- functions~\cite{hiller1986pieri,pragacz2006algebro}.

A major line of active research is to understand combinatorially more exotic cohomology theories such as $K$--theory, which encodes finer data about the boundaries of intersections.
For $K$--theory, products in the Grassmannian were first computed combinatorially by Buch~\cite{buch2002littlewood}, and later extended to the orthogonal Grassmannian~\cite{clifford2014k,buch2016k}.
For the Lagrangian Grassmannian, $K$--theory has proved far more difficult to understand combinatorially, with the only progress being Buch and Ravikumar's Pieri rule~\cite{buch2012pieri}.
We offer a new pathway towards understanding such products based on Type C $K$--Stanley symmetric functions.

\subsection{Stanley symmetric functions and Kra\'skiewicz--Hecke insertion}

The Stanley symmetric functions $F_w$ are symmetric functions indexed by permutations.
Introduced by Stanley to enumerate reduced words~\cite{stanley1984number}, Stanley symmetric functions are also stable limits of the Schubert polynomials $\fkS_w$, which represent cohomology classes in the flag variety $\mathrm{Fl}(n)$, and are themselves cohomology representatives for graph Schubert varieties~\cite{knutson2013positroid}.
Edelman-Greene insertion shows they expand into the Schur basis with non-negative coefficients enumerated by Edelman--Greene insertion tableaux~\cite{edelman1987balanced}, showing reduced words for permutations are enumerated by standard tableaux counts.
By choosing the permutation $w$ appropriately, this expansion recovers the Littlewood--Richardson rule for products of Schur functions.
Our work continues this line of investigation, introducing a novel insertion algorithm to better understand symmetric functions arising in Schubert calculus.

The symmetric functions we study are \defn{Type C $K$--Stanley symmetric functions} $G^C_w$, introduced in~\cite{kirillov2017construction} and indexed by signed permutations.
Our main tool is a novel insertion algorithm we call \defn{Kra\'skiewicz--Hecke insertion}, which generalizes Kra\'skiewicz insertion, the Type C analogue of Edelman--Greene insertion~\cite{kraskiewcz1989reduced}.
As shown in~\cite{lam1995b}, Kra\'skiewicz insertion interprets the Schur $Q$- coefficients of a Type C Stanley symmetric function from~\cite{billey1995schubert} as certain decomposition tableaux.
Consequently, reduced words for signed permutations are enumerated by standard shifted tableaux counts.
Hecke insertion is the $K$--theoretic extension of Edelman--Greene insertion from reduced words to all words in the alphabet $\{1,2,\dots\}$~\cite{buch2008stable}.
At the level of symmetric functions, Hecke insertion shows a $K$--Stanley symmetric function\footnote{Originally known as a Stable Grothendieck polynomial for a permutation.} from~\cite{fomin1994grothendieck} expands into stable Grothendieck polynomials\footnote{Hence our preference for the newer name.} with coefficients enumerating increasing tableaux.
This expansion gives a new proof for the $K$--theory product structure of the Grassmannian and shows 0--Hecke expressions for permutations are enumerated by standard set-valued tableaux counts.
The $0$--Hecke product and $0$--Hecke expressions are defined in Section~\ref{ss:signed-perms}.

Insertion algorithms map a word $(a_1,\dots,a_p)$ to an insertion tableau $P$ and a recording tableau $Q$.
The recording tableaux for Kra\'skiewicz--Hecke insertion are standard shifted set-valued tableaux from~\cite{ikeda2013k}.
We introduce \defn{strict decomposition tableaux} (see Definition~\ref{d:strict-decomposition}) to play the role of insertion tableaux.
For $\lambda$ a strict integer partition, let $\ShSet_n(\lambda)$ be the set of standard shifted set-valued tableaux containing $n$ values and $\SDT(\lambda)$ be the set of strict decomposition tableaux of shape $\lambda$.
Let $KH$ denote Kra\'skiewicz--Hecke insertion.
For $P$ a strict decomposition tableau, $\rho(P)$ is the usual reading word of $P$ (see Section~\ref{ss:tableaux}).

\begin{theorem}
\label{t:bijection}
	For all $n \in \mathbb{N}$, the map Kra\'skiewicz--Hecke insertion is a bijection:
	\[
KH:	  \mathbb{N}^n \to{\sim} \bigsqcup_{\lambda \vdash m \leq n \ \mathrm{strict}} \SDT(\lambda) \times \ShSet_n(\lambda).
	\]
	Moreover, for $KH(a_1,\dots,a_p) = (P,Q)$, the words $(a_1,\dots,a_p)$ and $\rho(P)$ are $0$--Hecke expressions for the same signed permutation.
\end{theorem}

The definition of Kra\'skiewicz--Hecke insertion and proof of Theorem~\ref{t:bijection} appear in Section~\ref{s:insertion} and is extraordinarily technical.
As opposed to most families of tableau, checking the strict decomposition tableau column condition for two entries requires examining the intermediate segment of the tableau's reading word.
As a consequence, the insertion rules can modify entries in two rows, with concomitant difficulties propagating throughout the arguments in our proofs.

For $w$ a signed permutation, let $\mathcal{H}_n(w)$ be the set of $0$--Hecke expressions for $w$ of length $n$.
Note $\mathcal{H}_{\ell(w)}$ is the set of reduced words $w$.
Also, let $a^C_w(\lambda)$ be the number of strict decomposition tableaux of shape $\lambda$ whose reading word is a 0--Hecke expression for $w$.
As an immediate consequence of Theorem~\ref{t:bijection}, we have:

\begin{corollary}
	\label{c:0-hecke-count}
	For $w$ a signed permutation and $n \in \mathbb{N}$, we have 
	\begin{equation}
	|\mathcal{H}_n(w)| =  \sum_{\lambda \vdash m \leq n\ \mathrm{strict}} a^C_w(\lambda) \cdot |\ShSet_n(\lambda) |.
\label{eq:0-hecke-count}		
	\end{equation}
	\end{corollary}

Corollary~\ref{c:0-hecke-count} is precisely analogous to the use of Edelman--Greene and Kra\'skiewicz insertions for reduced word enumeration and Hecke insertion for 0-Hecke expression enumeration in Type A.
All three of these insertions are enumerative shadows of ($K$--)Stanley symmetric function expansions.
Unfortunately, as shown in Section~\ref{ss:no-insertion} an insertion algorithm extending Kra\'skiewicz insertion \textbf{cannot} be used to directly compute the analogous expansion for $G^C_w$.
That said, we believe Kra\'skiewicz--Hecke insertion is correctly identifying coefficients.
For $\lambda$ a strict partition, $GQ_\lambda$ is a symmetric function defined in terms of shifted set-valued tableaux and represents $K$--theory classes for the Lagrangian Grassmannian~\cite{ikeda2013k}.

\begin{conjecture}
	\label{conj:expansion}
	For $w$ a signed permutation,
	\begin{equation}
	G^C_w = \sum_{\lambda \ \mathrm{strict}} \beta^{|\lambda| - \ell(w)}a^C_w(\lambda) \cdot GQ_\lambda.
\label{eq:GC-conj}		
	\end{equation}
\end{conjecture}

By combining~\cite[Lemma 7 (1)]{kirillov2017construction} with the main result from~\cite{chiu2023expanding}, it is known that $G^C_w \in \Z[\beta][GQ_\lambda: \lambda \ \mathrm{strict}]$.
However, establishing the positivity implied by Conjecture~\ref{conj:expansion} is an open problem, even using geometric methods.
By taking the coefficient of $x_1\dots x_n$ on both sides of~\eqref{eq:GC-conj}, we obtain Corollary~\ref{c:0-hecke-count} as a consequence of Conjecture~\ref{conj:expansion}.
Assuming the conjecture is false, replacing each $a^C_w(\lambda)$ with the correct coefficient of $GQ_\lambda$ for $G^C_w$ in~\eqref{eq:0-hecke-count} gives a true equation.
Therefore, should our conjecture fail we will have instead identified a striking collection of relations for standard shifted set-valued tableaux counts.

While we cannot establish Conjecture~\ref{conj:expansion}, we prove two important expansions of $GC$'s into $GQ$'s.
First, as a straightforward application of the Pfaffian formula for Type C degeneracy loci from~\cite{anderson2019k}\footnote{See the current arXiv version for the corrected formula}, we show:
\begin{theorem}
\label{t:vexillary}
	For $w$ a vexillary signed permutation, there is a shifted shape $\lambda(w)$ so that $G^C_w = GQ_{\lambda(w)}$.
\end{theorem}
Using the identification between reduced words and heaps for certain fully commutative elements~\cite{stembridge1996fully,stembridge97} (a similar and more general construction appears in~\cite{tamvakis2023tableau}), we also show:
\begin{theorem}
	\label{t:full-commutative}
	For $\nu/\mu$ a skew shifted shape, there is a signed permutation $w$ with $G^C_w = GQ_{\nu/\mu}$.
\end{theorem}

Many properties of $GQ$'s that are known for other families of $K$--theoretic symmetric functions remain open.
For example, there is no combinatorial rule for the product $GQ_\lambda \cdot GQ_\mu$.
By picking appropriate skew shapes, Theorem~\ref{t:full-commutative} and Conjecture~\ref{conj:expansion} combine to give conjectural descriptions for several $GQ$ expansions of interest.
This includes Conjecture 5.14 from~\cite{lewis2021enriched} and Conjecture 4.36 from~\cite{marberg2022shifted}, which state that a skew $GQ$ lies in $\Z[\beta][GQ_\lambda: \lambda \text{ strict}]$.
See Conjecture 74 from~\cite{hawkes2024combinatorics} for a distinct conjectural expansion.

Second, we give a conjectural rule for certain $GQ$ products, the first progress on this problem since Buch and Ravikumar's Pieri rule for computing $GQ_\lambda \cdot GQ_{(r)}$~\cite{buch2012pieri}.
For $a < b$, let 
\[\tau(a,b) = (b{+}a{-}1,b{+}a{-}3,\dots,b{-}a{+}1).
\]
Using Theorem~\ref{t:full-commutative}, we prove Conjecture 4.23 from~\cite{marberg2022shifted}, which allows us to show the product $GQ_\lambda \cdot GQ_{\tau(a,b)}$ is $GC_w$ for an appropriate signed permutation $w$.
Then Conjecture~\ref{conj:expansion} specializes to a combinatorial rule for such products.
Note $\tau(1,b) =(b)$ is a single row, so this would generalize the Buch--Ravikumar Pieri rule.
In forthcoming work, the first author reproves the Pieri case using strict decomposition tableaux.

\medskip

\subsection*{Paper structure}
Section~\ref{s:background} introduces necessary background material on signed permutations, set-valued tableaux and $GQ$ functions.
In Section~\ref{s:GC}, we give a precise definition for Type C $K$--Stanley symmetric functions and demonstrate their features required for our conjectures, including Theorems~\ref{t:vexillary} and~\ref{t:full-commutative}.
We introduce strict decomposition tableaux and prove several of their basic properties in Section~\ref{ss:strict-tableaux}.
The most technical material in our work is Section~\ref{s:insertion}, where we introduce Kra\'skiewicz--Hecke insertion and prove its properties leading to Theorem~\ref{t:bijection}.
% Section~\ref{s:fully-commutative} introduces a subset of fully commutative signed permutations whose heaps are skew shifted partitions, leading to a proof of Theorems~\ref{t:full-commutative}.
We discuss Conjecture~\ref{conj:expansion} and its applications including Conjecture~\ref{conj:trapezoid} in greater detail in Section~\ref{s:conjectures} before concluding with final remarks and further directions in Section~\ref{s:final}.

\bigskip

\section{Background}\label{s:background}
\smallskip
For $n$ a positive integer, let $\overline{n} = -n$, $[n] = \{1,2,\dots,n\}$ and $[\overline{n}] = \{\overline{1},\overline{2},\dots,\overline{n}\}$.
Let $\Z_-$ be the set of negative integers and $\Z_+$ be the set of positive integers.
Define $\prec$ as the total order on $\Z - \{0\}$ where $\overline{1} \prec 1 \prec \overline{2} \prec 2 \prec \dots$.

A \defn{strict partition} is a descending sequence $\lambda = (\lambda_1 > \lambda_2 > \dots > \lambda_k)$ of positive integers.
Here, the \defn{size} of $\lambda$ is $|\lambda| = \sum_{i = 1}^k \lambda_i$, and the \defn{length} of $\lambda$ is $\ell(\lambda) = k$.
The \defn{shifted Young Diagram} of $\lambda$ is the set
\[
D_\lambda = \{(i,j) \in \Z_+^2: i \leq j \leq \lambda_i + i\}.
\]
See Figure~\ref{fig:tableaux} for an example.
We say the strict partition $\mu$ is \defn{contained} in $\lambda$, denoted $\mu \subseteq \lambda$, if $\mu_i \leq \lambda_i$ for all $i \in [\ell(\mu)]$, or equivalently if $D_\mu \subseteq D_\lambda$.
Here, we refer to $\lambda/\mu$ as a \defn{skew shape}, whose associated diagram is $D_{\lambda/\mu} = D_\lambda \setminus D_\mu$.

\subsection{Signed Permutations}
\label{ss:signed-perms}
A \defn{signed permutation} $w$ is a permutation of the elements $[\overline{n}]\cup[n]$ such that $w(i)=-w(\overline{i})$.
For example, $v = \overline{1}\overline{3}2\overline{2}31$ is a signed permutation.
By antisymmetry, $w$ is determined by $w([n])$.
When writing a signed permutation, we frequently omit $w([\overline{n}])$; for instance $v = \overline{2}31$.
Signed permutations with composition form a group $W_n$, which is the Coxeter group of Type B/C.
Viewed as a Coxeter group, the generators of $W_n$ are $s_0, s_1, \ldots s_{n-1}$ where $s_0 = (\overline{1},1)$ swaps positions $\overline{1}$ and $1$ and $s_i = (\overline{i{+}1},\overline{i})(i,i{+}1)$ simultaneously swaps positions $i$ and $i+1$ and positions $\overline{i}$ and $\overline{i+1}$ for $i>0$.
These generators satisfy the relations
\begin{enumerate}
    \item Self-inverse: $s_i^2 = 1$ for $i =0,1,\dots,n-1$;
    \item Commutation: $s_i s_j = s_i s_j$ if $|i-j| \geq 2$;
    \item Braid Relation: $s_i s_{i+1} s_i = s_{i+1} s_i s_{i+1}$ if $i > 0$;
    \item Long Braid Relation: $s_1 s_0 s_1 s_0 = s_0 s_1 s_0 s_1$.
\end{enumerate}

The \defn{length} of $w \in W_n$ is the minimum number of generators needed to express $w$, denoted $\ell(w)$.
A \defn{reduced expression} for $w$ is an expression of the form
\[
w = s_{a_1} \cdot s_{a_2} \cdot \ldots \cdot s_{a_p}
\]
where $p = \ell(w)$.
The associated word $(a_1,\dots,a_p)$ is a \defn{reduced word} for $w$.
The Matsumoto--Tits theorem says that the set of reduced words for a given $w$ is connected by the commutation, braid and long braid relations.
For example, the reduced words of $w = \overline{2}31$ are $(1,0,2)$ and $(1,2,0)$.

The \defn{0--Hecke monoid} $(W_n,\circ)$ is the monoid on signed permutations obtained by replacing the Self-inverse relation with the Idempotent Relation $s_i \circ s_i = s_i$.
A \defn{0--Hecke expression} for $w \in W_n$ is an expression of the form
\[
w = s_{a_1} \circ s_{a_2} \circ \ldots \circ s_{a_p}
\]
with associated \defn{Hecke word} $(a_1,a_2,\dots,a_p)$.
Let $\HH_p(w)$ be the set of Hecke words for $w$ with $p$ letters and $\HH(w) = \cup_{p \geq 0} \HH_p(w)$.
For example, with $w = \overline{2}31$ we have $\HH_3(w) = \{(1,0,2),(1,2,0)\}$ and
\[
\HH_4(w) = \{(1,0,0,2),(1,0,2,0),(1,0,2,2),(1,1,0,2),(1,2,0,0),(1,2,0,2),(1,2,2,0)\}.
\]
Note $\HH_{\ell(w)}(w)$ is the set of reduced words for $w$ and $\HH_m(w)$ is empty when $m < \ell(w)$.
The \defn{peak set} of a word $(a_1,\dots,a_p)$ is $\Peak((a_1,\dots,a_p)) = \{i \in [p]: a_{i-1} < a_i > a_{i+1}\}$.
By construction $1,p \notin \Peak((a_1,\dots,a_p))$.\textbf{}

We now introduce several important families of signed permutations.
A signed permutation $w$ is \defn{Grassmannian} if $w(i) < w(i+1)$ for all $i\in [n-1]$, that is $w_1, \dots, w_n$ is an increasing sequence.
Each Grassmannian signed permutation is associated with a strict partition, $\lambda$, defined by $\lambda_i = \overline{w(i)}$ for each $i$ such that $w(i) < 0$.
For example, $w = \overline{4}\overline{1}23$ is a Grassmannian signed permutation associated with $\lambda = (4,1)$.
Note the identity corresponds to the empty partition.

Let $w\in W_n$ and $v\in V_m$, $w$ contains $v$ as a \defn{pattern} if there exists a subsequence of $w$ of length $m$, $w'$, such that for all $i,j\in [m]$, $w'(i)$ and $v(i)$ have the same sign and $|w'(i)|< |w'(j)|$ if and only if $|v(i)| < |v(j)|$. If $w$ does not contain $v$ as a pattern, then $w$ \defn{avoids} $v$. A signed permutation, $w$, is \defn{vexillary} if $w$ avoids the following $18$ patterns \cite{billey1998vexillary}:
\[
\begin{matrix}
 \overline{1}32 &  \overline{2}31 & 3\overline{1}2 & 321 & \overline{3}21 & 32\overline{1} & \overline{3}2\overline{1}\\
 2143 & 2\overline{3}4\overline{1} & \overline{2}\overline{3}4\overline{1} & 2413 & 3142 & 3\overline{4}1\overline{2} & \overline{3}\overline{4}1\overline{2}\\ 3\overline{4}\overline{1}\overline{2} & \overline{3}\overline{4}\overline{1}\overline{2} & \overline{4}1\overline{2}3 & \overline{4}\overline{1}\overline{2}3
\end{matrix}
\]
Each vexillary permutation has an associated shifted shape $\lambda(w)$, which we define implicitly in Section~\ref{ss:symmetric-functions}, see \cite{billey1998vexillary} for an algorithmic construction of $\lambda(w)$.

\subsection{Shifted Tableaux}
\label{ss:tableaux}
For $\lambda$, $\mu$ strict partitions with $\mu \subseteq \lambda$, a \defn{shifted tableau} of \defn{shape} $\lambda/\mu$ is a function $T$ whose domain is $D_{\lambda/\mu}$.
Here, $T_{ij} = T((i,j))$ is called the \defn{entry} in the cell $(i,j)$.
For a shifted tableau $T$, let $T_i$ be the entries in the $i$\textsuperscript{th} row of $T$, read from left to right.
The \defn{reading word} of $T$ with $\ell$ rows is $\rho(T)=T_{\ell} T_{\ell-1}\ldots T_1$.
A \defn{shifted standard Young tableau} is a shifted tableau filled bijectively by $[n]$ so that each row is strictly increasing rightwards and each column is strictly increasing downwards.
Viewing $D_{\lambda/\mu}$ as a poset with $(i,j) \leq (k,\ell)$ if $i \leq k$ and $j \leq \ell$, note shifted standard Young tableaux of shape $\lambda/\mu$ are in bijection with linear extensions of $D_{\lambda/\mu}$.

A \defn{set-valued tableau} $T$ has entries that are finite sets of integers, that is $T_{ij} \subseteq \Z$.
Similarly, a \defn{shifted set valued standard Young tableau} $T$ of \defn{size} $|T| = n$ is a shifted diagram filled with sets that partition $[n]$ such that for all cells
\[
\max T_{ij} \leq \min T_{i+1\,j}, \min T_{i\,j+1}
\]
(assuming such cells exist).
Let $\ShSet(\lambda)$ be the set of shifted set valued standard Young tableaux of shape $\lambda$ and $\ShSet_n(\lambda)$ be the subset whose entries partition $[n]$.
Note for $\lambda\vdash n$ that $\ShSet_n(\lambda)$ is the set of shifted standard Young tableaux of shape $\lambda$.
The \defn{peak set} of a shifted set-valued standard Young tableau $T$ is $\Peak(T)$, the set of values $i$ such that $i-1$ is in a column to the left of $i$ and $i+1$ is is in a row below $i$.
Note for any tableau $T$ that $i \in \Peak(T)$ implies $i-1, i+1 \notin \Peak(T)$.
See Figure~\ref{fig:tableaux} for examples.

A \defn{shifted set valued semistandard Young tableau} is a set valued tableau with entries in $\Z - \{0\}$ so that $\max(T_{ij}) \preceq \min(T_{i{+}1\,j})$ with equality only for negative values and $\max(T_{ij}) \preceq \min(T_{i\,j{+}1})$ with equality only for positive values.
For $T$ a set valued tableau, define
\[
x^T = \prod_{(i,j) \in \lambda(T)} x^{T_{ij}} \quad \mbox{where} \quad x^{S} = x_{|s_1|} \dots x_{|s_k|} \quad \mbox{for} \quad S = \{s_1,\dots,s_k\} \subseteq \Z.
\]
Let $\ShSet^*(\lambda/\mu)$ be the set of shifted set valued semistandard Young tableaux of shape $\lambda/\mu$.
Then
\begin{equation}
\label{eq:GQ}
GQ^{(\beta)}_{\lambda/\mu} = \sum_{T\in \ShSet^*(\lambda/\mu)} \beta^{|T| - |\lambda/\mu|} x^T.    
\end{equation}
We write $GQ^{(\beta)}_\lambda := GQ^{(\beta)}_{\lambda/\varnothing}$.
For $\mu = \varnothing$, the $GQ$'s were first introduced in~\cite{ikeda2013k} (see also~\cite{graham2015excited}) as $K$--theory representatives for Schubert classes in the Lagrangian Grassmannian.
The lowest degree homogeneous component of $GQ_\lambda$ is the \defn{Schur $Q$-function} $Q_\lambda = GQ_\lambda^{(0)}$.
While it is not obvious from the above definition, the $GQ^{(\beta)}_\lambda$'s are symmetric functions in the $x$--variables.

Let $\lambda$ be a strict partition with $\ell(\lambda) = k$ and $\mu = (k-1,\dots,1)$.
Then the shifted diagram of $\lambda/\mu$ is the ordinary diagram of the (non-strict) partition $\nu = (\lambda_1 - (k{-1}), \dots \lambda_k)$.
An important special case of skew $GQ$ functions are the $GS$ functions defined by $GS_{\nu} = GQ_{\lambda/\mu}$.
These are generating functions for unshifted marked set-valued tableaux.

\begin{figure}[h!]
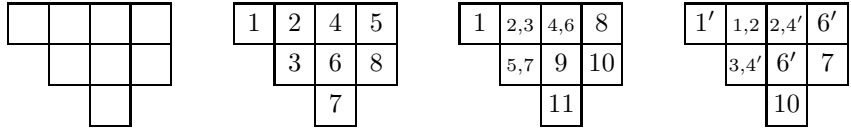

    \centering
    \begin{ytableau}
        \; & \; & \; & \;\\
         \none & \; & \; & \;\\
         \none & \none & \;
    \end{ytableau} \quad \quad
    \begin{ytableau}
        1 & 2 & 4 & 5 \\
         \none & 3 & 6 & 8\\
         \none & \none & 7
    \end{ytableau}
    \quad \quad
    \begin{ytableau}
        1 & {\scriptstyle 2,3} & {\scriptstyle 4,6} & 8 \\
         \none & {\scriptstyle 5,7} & 9 & 10\\
         \none & \none & 11
    \end{ytableau}
    \quad \quad
    \begin{ytableau}
        1' & {\scriptstyle 1,2} & {\scriptstyle 2,4'} & 6'\\
        \none & {\scriptstyle 3,4'} & 6' & 7\\
        \none & \none & 10
    \end{ytableau}
    
    \caption{To the left is a shifted Young diagram of shape $\lambda=(4,3,1)$. Next is a standard shifted Young tableau $T_1$ of the same shape with peak set $\Peak(T_1) = \{2,5\}$, followed by a standard shifted set valued Young tableau $T_2$ of the same shape with $\Peak(T_2) = \{4,6,8,10\}$.
    Last is a semistandard shifted set valued Young tableau $T_3$ so that $\st(T_3) = T_2$. 
    \label{fig:tableaux}}
\end{figure}

Given a set-valued semistandard tableau $T$, we construct the \defn{standardization} of $\st(T)$ iteratively as follows.
First assume $T$ has only entries $\overline{a}$ and $a$, with $|T| = k$ and $a \notin [k]$.
Beginning as high as possible, replace each $\overline{1}$ with the smallest value of $[k]$ not in $T$.
After all such values have been replaced, beginning as low as possible replace each $1$ with the smallest value of $[k]$ not in $T$.
The resulting tableau is standard.
For a set-valued semistandard tableau whose largest entry is $m$ or $\overline{m}$, we can repeat this procedure for each value in $[m]$ to construct a standard tableau of the same size.

\section{Type $C$ $K$--Stanley symmetric functions}
\label{s:GC}
In~\cite{kirillov2017construction}, the authors introduce Type $C$ $K$--Stanley symmetric functions $G^C_w$ as part of their construction of Type C double Grothendieck polynomials.
However, as ancillary objects in their work the $G^C_w$'s were not investigated for their own sake.
Since then these symmetric functions have not received further study, so we must establish several of their basic properties.

\subsection{Definition and basic properties}
\label{ss:symmetric-functions}
Let $\bfx = (x_1,x_2, \dots)$ be commuting variables and $u_1,\dots, u_{n-1}$ be non-commuting variables subject to the following $\beta$--Idempotent Relation and the Coxeter relations:
\begin{align*}\label{noncommutative}
&u_i^2 = \beta u_i, &u_1u_0u_1u_0=u_0u_1u_0u_1,\\
&u_iu_{i+1}u_i=u_{i+1}u_iu_{i+1}\  (i>0), &u_iu_j =u_ju_i\  (|i-j|>1).
\end{align*}
Define
\begin{equation}
    \label{eq:gc-def}
G^C_n(\bfx) = \prod_{i\geqslant 1}(1+x_iu_{n-1})(1+x_iu_{n-2})\dots(1+x_iu_1)(1+x_iu_0)(1+x_iu_0)(1+x_iu_1)\dots (1+x_iu_{n-1}),
\end{equation}
where the product is computed beginning with the terms containing $x_1$.
For $w\in W_n$, the \defn{type $C$ K-Stanley symmetric function} of $w$ is the coefficient of $u_w$ in $G^C_n(\bfx)$, which we denote $G^C_w$.
These symmetric functions (with $\beta = -1$) are introduced in~\cite{kirillov2017construction} in order to define Type C double Grothendieck polynomials and are stable under the natural inclusion $W_n \hookrightarrow W_{n+1}$.
Note that $G^C_w$ is not homogeneous; its lowest degree homogeneous component is the \defn{Type C Stanley symmetric function} $F^C_w = G^C_w \mid_{\beta =0}$.
It is not obvious, but $G^C_w$ is symmetric for each $w$.
In fact, it can be expressed in terms of $GQ$--functions.
\begin{proposition}
    \label{p:GC-GQ}
    For $w \in W_n$, $G^C_w \in \Z[\beta][GQ^{(\beta)}_\lambda:\lambda$ strict$]$.
\end{proposition}
\begin{proof}
    This is essentially~\cite[Lemma 7~(1)]{kirillov2017construction}, which says $G^C_w \in \Z[\beta][GP^{(\beta)}_\lambda:\lambda$ strict$]$, where $GP^{(\beta)}_\lambda$ is a symmetric function closely related to the $GQ^{(\beta)}_\lambda$ that we do not define.
    In~\cite{chiu2023expanding} the authors show each $GP^{(\beta)}_\lambda$ is an element of  $\Z[GQ^{(\beta)}_\lambda: \lambda$ strict$]$, so the result follows.
\end{proof}

A \defn{unimodal factorization} of the word $\bfa  = (a_1, \dots, a_p)$ in the alphabet $\mathbb{N}$ is the biword $(\bfa,\bfi)$ where
\[
\bfi = i_1 \preceq i_2 \preceq \dots \preceq i_p
\]
with $i_j \in \Z - \{0\}$ for all $j$ so that $i_k = i_{k+1}$ implies $a_k > a_{k+1}$ when $i_k < 0$ and $a_k < a_{k+1}$ when $i_k > 0$.
Such biwords are also known as compatible sequences.
Our nomenclature comes from the following depiction of the unimodal factorization $(\bfa,\bfi)$: group consecutive subwords of $\bfa$ whose corresponding entry in $\bfi$ have the same absolute value using parentheses, indicating the transition from negative to positive values in each parenthesized subword with $|$.
For example, the unimodal factorization
\[
\left(\bfa = (0,1,2,3,0,1,4,2,3,3,0,1,2,5,0,1), \quad \bfi = (1,1,1,\overline{2},2,2,\overline{4},\overline{4},4,\overline{5},5,5,\overline{6},6,6)\right)
\]
is depicted as
\[
(|\; 012)(3\;|\;01)()(42\;|\;3)(3\;|\:012)(5\;|\;01).
\]
Let $\UU(w)$ be the set of unimodal factorizations whose word is in $\HH(w)$ and $\UU_p(w)$ be the subset whose word is in $\HH_p(w)$.
By interpreting the definition of $G^C_w(\bfx)$ in terms of~\eqref{eq:gc-def} appropriately, we have:
\begin{proposition}[{\cite[Prop.~17]{kirillov2017construction}}]
    \label{p:unimodal_GC}
    For $w$ a signed permutation,
    \[
    G^C_w(\bfx) = \sum_{p \geq 0} \sum_{(\bfa,\bfi) \in \UU_p(w)} \beta^{p-\ell(w)} x^\bfi \quad \mbox{where} \quad x^\bfi = x_{|i_1|} \dots x_{|i_p|}.
    \]
\end{proposition}

\begin{proof}
    A non-zero monomial in $G^C_n(\bfx)$ is of the form
    \[
    x_{j_1} u_{a_1} \cdot \ldots \cdots x_{j_p} u_{a_p}.
    \]
    Construct the biword $(\bfa,\bfi)$ from this monomial with $\bfa = (a_1,\dots,a_p)$ and $\bfi = (i_1,\dots,i_p)$ where 
    \[
    i_k = \begin{cases}
        \overline{j_k} & x_{i_k} u_{a_k}\ \mbox{is in the descending part of the}\ j_k th\ \mbox{term in}\ G^C_n(\bfx),\\
        j_k. & \mbox{else}
    \end{cases}
    \]
    The reader can confirm that $(\bfa,\bfi) \in \UU_p(w)$ where $w = u_{a_1} \circ \dots \circ u_{a_p}$, and that all such terms arise in this way.
    \end{proof}

As a corollary, we see Type $C$ $K$--Stanley symmetric functions can be used to enumerate $0$--Hecke expressions.

\begin{corollary}
    \label{c:hecke-enumeration}
    For $w \in W_n$, the coefficient of $\beta^{n-\ell(w)}x_1 \dots x_n$ in $G^C_w$ is $2^n \cdot |\HH_n(w)|$.
\end{corollary}

\begin{proof}
    This follows Proposition~\ref{p:unimodal_GC}: each monomial $x_1 \dots x_n$ corresponds to a unimodal factorization $(\bfa,\bfi) \in \UU(w)$ where $\bfa \in \HH_n(w)$ and each $i_j$ can be $\overline{j}$ or $j$ independent of other choices.
\end{proof}

Using the unimodal factorization characterization of Type C $K$--Stanley symmetric functions, we can characterize certain products, generalizing and extending~\cite[Cor.~3.3]{billey1998vexillary} to the $K$--theoretic setting.

\begin{proposition}
    \label{p:signed-cross-ordinary}
    Let $u \in W_k$ and $v \in S_n$ so that $v(i) = i$ for all $i \in [k]$. Then
    \[
    G^C_{uv} = G^C_u \cdot G^C_v.
    \]
\end{proposition}

\begin{proof}
    Consider the $n$th parenthesization of $(\bfa,\bfi) \in \UU_{uv}$, which is of the form
    \[
    (b_1 > \dots > b_i > c_1 > \dots > c_j \mid c_{j+1} < \dots < c_\ell < b_{i+1} < \dots < b_m).
    \]
    where the $b'$ are greater than $k$ and $c$'s are less than $k$.
    This term can be split into two parenthesizations
    \[
    (b_1 > \dots > b_i \mid b_{i+1} < \dots < b_m) \quad \mbox{and} \quad (c_1 > \dots > c_j \mid < c_{j+1} < \dots < c_\ell).
    \]
    Applying this procedure to all parenthesizations of $(\bfa,\bfi)$, we obtain unimodal factorizations $(\textbf{b}, \textbf{j}) \in \UU(u)$ and $(\textbf{c},\textbf{l}) \in \UU(w)$, respectively.
    This process is invertible, so we have a weight preserving bijection from $\UU(uv)$ to $\UU(u) \times \UU(v)$, hence the result follows from Proposition~\ref{p:unimodal_GC}.
\end{proof}

In~\cite{kirillov2017construction}, the authors show for Grassmannian $w(\lambda)$ that $G^C_{w(\lambda)} = GQ^{(\beta)}_\lambda$ via a geometric argument.
This is also a special case of our Theorem~\ref{t:full-commutative}, which we prove combinatorially.
We extend the Grassmannian result to all vexillary permutations via a geometric argument:

\begin{proposition}
    \label{p:GC-vexillary}
    For $w \in W_n$ vexillary, $G^C_w = GQ^{(\beta)}_{\lambda(w)}$.
\end{proposition}

\begin{proof}[Proof sketch]
For the degeneracy locus associated to a vexillary signed permutation  $w$, there is a Pfaffian formula for its equivariant $K$--theory class in terms of Chern classes $c_w(j)$~\cite[Thm.~2]{anderson2019k}.
With appropriate conventions and specializing to the single case, we can view such Chern classes as
\[
c_w(j) =  \cdot \frac{(1-y_a)}{(1-z_b)} \cdot (1 + c_1 + c_2 + \dots)
\]
where $y_a$ and $z_b$ are variables whose indices are determined by $w$.
The terms of the matrix whose Pfaffian is being taken are certain $c_w(j)$'s, where each $j$ term is determined by the shape $\lambda(w)$.
By~\cite[Prop.~2]{kirillov2017construction}, setting the $y_i$ and $z_j$ variables to zero recovers $G^C_w$ from this formula.
Therefore $G^C_w$ is determined by $\lambda(w)$, so the result follows from the Grassmannian case.
\end{proof}
% \Zach{
% I want to add ``Since the details are quite technical, we defer them to Appendix~??", which would be quite unpleasant to write. Perhaps we can get away with omitting them?
% }

Note by taking $\beta=0$ in (2) we have for $w$ vexillary that $F^C_w = Q_{\lambda(w)}$, as shown in~\cite{billey1998vexillary}.
We take this as the definition of $\lambda(w)$.

For $\lambda$ a strict partition and $a, b$ positive integers, recall from the introduction that 
\[
w(a,b,k) = 12 \dots k\,\ell{+}1\, \ell{+}2\dots n\,k{+}1\,k{+2}\dots \ell
\]
where $\ell = a + k$ and $n = a + b + k$.
When $a \leq b$, recall also the shifted trapezoid
\[
\tau(a,b) = (a+b-1,a+b-3,\dots b-a+1).
\]
As a consequence of Proposition~\ref{p:GC-vexillary} we prove Conjecture~4.23 from~\cite{lewis2021enriched}:
\begin{corollary}
    \label{c:rectangle-trapezoid}
For $a,b,k$ positive integers with $a\leq b$,
\[
G^C_{w(a,b,k)} = GQ_{\tau(a,b)}.
\]
    
\end{corollary}
\begin{proof}
By discussion after~\cite[Cor.~3.3]{billey1998vexillary}, we see $G^C_{w(a,b,k)}$ is vexillary of shape $\tau(a,b)$.
Therefore, the result follows by Proposition~\ref{p:GC-vexillary}.
\end{proof}

From this, we have:

\begin{corollary}
\label{c:GQ-trapzeoid}
    For $\lambda$ a strict partition with $\lambda_1 < k$ and $a \leq b$ positive integers,
\[
G^C_{w(\lambda,k)\cdot w(a,b,k)} = GQ_\lambda \cdot GQ_{\tau(a,b)}.
\]
\end{corollary}

\begin{proof}
    We compute
    \[
    G^C_{w(\lambda,k) \cdot w(a,b,k)} = G^C_{w(\lambda,k)} \cdot G^C_{w(a,b,k)} = GQ_\lambda \cdot GQ_{\tau(a,b)}
    \]
    with the first equality by Proposition~\ref{p:signed-cross-ordinary} and the second by Proposition~\ref{p:GC-vexillary} and Corollary~\ref{c:GQ-trapzeoid}.
\end{proof}

One might hope further $GQ$--products can be modeled as $GC$--expansions.
Unfortunately, as discussed in~\cite{billey1998vexillary}, trapezoids are the \textbf{only} shifted shapes $\lambda$ for which $w \in S_n$ exists with $GQ_w|_{\beta = 0} = Q_\lambda$.
This property was needed to ensure $w(\lambda,k)$ would not be changed in above product.

\subsection{Top fully commutative elements}
In this section, we study top fully commutative elements of $W_n$, which we abbreviate to `top' and are indexed by skew shifted shapes.
There is a transparent bijection between shifted standard tableaux of strict shape $\lambda/\mu$ and reduced words for the top element $w_{(\lambda/\mu)}$ due to Stembridge~\cite{stembridge1996fully,stembridge97}.
We extend this bijection to the map $\res: \ShSSYT_p(\lambda/\mu) \to{} \UU_p(w_{(\lambda/\mu)})$, proving Theorem~\ref{t:full-commutative}, which says
\[
GQ^{(\beta)}_{\lambda/\mu} = G^C_{w(\lambda/\mu)}
\]

We say $w \in W_n$ is \defn{fully commutative} if its reduced words contain no braid relations, that is none of them contain the consecutive subwords $(0,1,0,1)$ and $(i,i+1,i)$ for $i \geq 1$.
Fully commutative elements for Coxeter groups have been studied extensively, with Stembridge~\cite{stembridge1996fully} giving the first systematic treatment.
A key feature of fully commutative $w$ is that its reduced words are in bijection with linear extensions of an associated poset called the \defn{heap} of $w$.

We are especially interested in signed permutations $w$ that are \defn{top fully commutative}, abbreviated to \defn{top}, which are fully commutative and whose reduced words do not contain the consecutive subword $(1,0,1)$.
As shown in~\cite[Cor.~5.6]{stembridge97}, top signed permutations are those avoiding the patterns
\[
\overline{1}\overline{2},\ \overline{1}2,\ \overline{3}2\overline{1},\ \overline{3}21,\ 32\overline{1}, 321.
\]
The heaps of top elements are skew strict partitions.

\begin{proposition}[{\cite[Thm.~8.2]{fomin1996combinatorial},\ \cite[Prop.~6.1]{stembridge97}}]  
\label{p:top-heap}
For $w \in W_n$ top, there is a shifted shape $\lambda/\mu$ so that the heap of $w$ is isomorphic to $D_{\lambda/\mu}$.
Moreover, for every skew shape $\lambda/\mu$ such a top element exists.
\end{proposition}
Proposition~\ref{p:top-heap} is stated without proof in~\cite{fomin1996combinatorial}.
A proof appears in~\cite{stembridge97}, where a complete characterization is presented for the fully commutative elements that are not top but whose heap is isomorphic to $D_{\lambda/\mu}$ for some $\lambda/\mu$.

The Grassmannian permutation of shape $\lambda$ is the top element whose heap is the reverse (as a poset) of $D_\lambda$.
More generally, for shifted shape $\lambda/\mu$ with $|\lambda/\mu| = p$ and $\ell(\lambda) = k$, we can construct the associated top element $w_{(\lambda/\mu)}$ as follows.
Let $C$ be the tableau of shape $\lambda/\mu$ whose $(i,j)$th entry $C_{ij} = j - i$ is the \defn{content} of cell $(i,j)$.
Note contents are non-negative since $\lambda/\mu$ is shifted.
Then the top-to-bottom reading word $C_1 C_2 \dots C_k$ of $C$ is the \defn{canonical reduced word} for the top element $w_{(\lambda/\mu)}$ whose heap is $D_{\lambda/\mu}$.
This construction is equivalent to the alternate characterization of top elements~\cite[Cor.~5.6~(b)]{stembridge97}.

The construction of $w_{(\lambda/\mu)}$ is a special case of the bijection between linear extensions of $D_{\lambda/\mu}$, i.e., shifted standard Young tableaux of shape $\lambda/\mu$, and reduced words for $w_{(\lambda/\mu)}$.
Following~\cite{morse2020crystal}, we call this map (and its subsequent generalization) the \defn{residue map}, denoted $\res$.
Given a shifted standard Young tableau $T$ of shape $\lambda/\mu$, construct the associated reduced word $\res(T) = (a_1,\dots,a_p)$ for $w_{(\lambda/\mu)}$ by setting $a_k = j-i$ where $(i,j) = T^{-1}(k)$ is the cell labeled $k$ in $T$.
For example, with $\lambda/\mu = (65421/41)$ the top element $w_{(\lambda/\mu)}$ has canonical reduced word $(4,5,1,2,3,4,0,1,2,3,0,1,0)$, and
\[
T = \begin{ytableau}
    \none & \none & 2 & 3 & 6\\
     \none & 1 & 4 & 7\\
    \none & \none &5
\end{ytableau} \quad \mbox{has} \quad \res(T) = (0,2,3,1,0,4,2) \in \HH_7(w_{(531/2)}).
\]

The above description of the residue map is the restriction of a more general map $\res$ from shifted set-valued semistandard tableaux to unimodal factorizations of $0$--Hecke expressions.
\begin{definition}
\label{d:res}
For $T \in \ShSSYT(\lambda/\mu)$, construct $\res(T)$ iteratively as follows.
For each positive integer $k$ in ascending order, we construct the $k$th parenthesization $(\bfa^k\  \mid \ \bfb^k)$ as follows:
\begin{itemize}
    \item Let $(i_1,j_1),\dots, (i_\ell,j_\ell)$ be the cells containing $\overline{k}$ with $i_1 < \dots < i_\ell$.
    Then $\bfa^k = (j_1 {- i_1},\dots, j_\ell {-} i_\ell)$.
    \item Let $(p_1,q_1),\dots, (p_m,q_m)$ be the cells containing $k$ with $q_1 < \dots < q_m$.
    Then $\bfb^k = (q_1 {-} p_1,\dots,q_m {-} p_m)$.
\end{itemize}
Note the semistandard conditions guarantee that each row contains at most one $\overline{k}$ and each column contains at most one $k$ so the above construction is well-defined.
To see the resulting parenthesization is unimodal, note the cells containing $\overline{k}$ are being read weakly from right to left, strictly from top to bottom so their contents are descending.
Likewise, the cells containing $k$ are being read strictly from left to right, weakly from bottom to top so their contents are increasing.

When $T$ is standard with $|T| = p$, we see each letter is in the positive part of its own parenthesization so $\res(T) = (\bfa,\bfi)$ with $\bfi = (1,\dots,p)$.
Then $T \mapsto \bfa$ is restriction of $\res$ to a map from $\ShSet_p(\lambda/\mu)$ to $\N^p$.
In this guise, $\res$ generalizes the aforementioned bijection from linear extensions of the heap to reduced words of $w_{\lambda/\mu}$.
\end{definition}

We will show for $T \in \ShSSYT(\lambda/\mu)$ that $\res(T)$ is lies in $\UU(w_{(\lambda/\mu)})$.
For example, with
\[
T = 
\begin{ytableau}
1_{\color{red}{0}} & 1_{\color{red}{1}} & 1_{\color{red}{2}} & 2'_{\color{red}{3}} & 3'_{\color{red}{4}} & 5'_{\color{red}{5}}\\
\none & 2_{\color{red}{0}} & 2_{\color{red}{1}} & 3'_{\color{red}{2}} & 3 4'_{\color{red}{3}}\\
\none & \none & 4_{\color{red}{0}} & 4_{\color{red}{1}} & 4_{\color{red}{2}}\\
\none & \none & \none & 5_{\color{red}{0}} & 5_{\color{red}{1}}
\end{ytableau}\,, \quad  \res(T) = (\mid 012)(3\mid 01)(42\mid 3)(3\mid 012)(5\mid 01).
\]
Here, the contents are depicted using red subscripts inside of $T$ for reference.

We require a simple lemma about $0$--Hecke expressions of fully commutative elements, which we phrase for top elements.
\begin{lemma}
    \label{l:top-word}
    For $\lambda/\mu$ a skew shifted shape, let $\bfa = (a_1,\dots,a_p) \in \HH(w_{(\lambda/\mu)})$, and let $a_i = a_\ell = m$ so that $m \notin \{a_{i+1},\dots,a_{\ell-1}\}$.
    Then 
    \[
    s_{a_1} \circ \dots \circ s_{a_{\ell- 1}} = s_{a_1} \circ \dots \circ s_{a_{\ell- 1}} \circ s_{a_\ell}
    \]
    if and only if for $j \in (i,\ell)$, we have $a_j \neq m+1,m-1$.

\end{lemma}

\begin{proof}
    For the forward direction, since all $0$--Hecke expressions are connected by the Commutation and Idempotent relations, we see we must be able to commute $a_\ell$ up to $a_i$ and cancel, hence no intermediate values can be present to prevent this commutation.
    For the converse, if no such $a_j$ exists then we can commute $s_{a_\ell}$ up to $s_{a_i}$ and apply the Idempotent relation.
\end{proof}

\begin{theorem}
\label{t:res}    
    The map $\res:\ShSSYT(\lambda/\mu) \to{} \UU(w_{(\lambda/\mu)})$ is a weight preserving bijection, that is for $\res(T) = (\bfa,\bfi)$ we have $x^T = x^\bfi$.
\end{theorem}

\begin{proof}
    Note the power of $x_k$ in $x^T$ is the number of $\overline{k}$ and $k$'s in $T$, while the power of $x_k$ in $x^\bfi$ is the size of the $k$th parenthesization in $(\bfa,\bfi)$.
    By construction, these quantities are the same so $\res$ is weight preserving.

    To see $\res$ is a bijection, we first treat the case where $T \in \ShSet_p(\lambda/\mu)$.
    In this case 
    \[
    \res(T) = (\bfa,\bfi = (1,2,\dots,p))
    \]
    by construction, so we must show the restriction $T \mapsto \bfa$ is a bijection from $\ShSet_p(\lambda/\mu)$ to $\HH_p(w_{\lambda/\mu})$.
    We first show this restriction is well-defined.
    Let $m_1 < \dots < m_k$ be the minimum values of the cells of $T$, occurring in cells $(i_1,j_1), \dots (i_k,j_k)$.
    Then the subword
    \[
    (a_{m_1},\dots,a_{m_k}) = (j_1 - i_1,\dots,j_k-i_k)
    \]
    is a reduced word for $w_{\lambda/\mu}$ since $\res$ restricts to a bijection from standard tableaux to reduced words.
    For $r \in [p] \setminus \{m_1,\dots,m_k\}$ in cell $(i,j)$ also containing $m_\ell$, we see the restriction $T \mid_{[r]}$ cannot contain cells $(i,j+1)$ or $(i+1,j)$.
    Therefore, in $\res(T) = (a_1,\dots,a_p)$ we have $a_r = j-i  = a_{m_\ell}$ and all intermediate values $a_{m_\ell + 1}, \dots, a_{r-1}$ cannot be $j-i-1, j-i+1$ as these would be drawn from cells $(i,j-1)$ and $(i-1,j)$, respectively, violating the semistandard condition.
    By Lemma~\ref{l:top-word}, we see 
    \[
    s_{a_1} \circ \dots \circ s_{a_{r-1}} = s_{a_1} \circ \dots \circ s_{a_{r-1}} \circ s_{a_r},
    \]
    so $\res(T) \in \HH_p(w_{\lambda/\mu})$.

    We now demonstrate an inverse map for the restriction $T \mapsto \bfa$.
    For $\bfa \in \HH_p(w_{\lambda/\mu})$, let $m_1< \dots < m_k$ be all values so that
    \[
    s_{a_1} \circ \dots \circ s_{m_1 - 1} < s_{a_1} \circ \dots \circ s_{m_1 - 1} \circ s_{m_1}.
    \]
    Then $\bfa' = (a_{m_1},\dots,a_{m_k})$ is a reduced word for $w_{\lambda/\mu}$.
    Since $\res$ restricted to standard tableaux and reduced words is a bijection, we see $\bfa'$ is the image of some $T \in \SYT(\lambda/\mu)$ under $\res$ with $i$th entry replaced with $m_i$.
    We fill in the rest of $T$ iteratively so that it is set-valued standard.
    For $\ell$ minimal amongst all values not yet in $T$, let $m_i$ be maximal so that $m_i < \ell$  and $a_{m_i} = a_\ell$.
    By Lemma~\ref{l:top-word}, we see for $m_i < j < \ell$ that $a_j \neq a_\ell \pm 1$ so $m_i$ is in an outer corner of $T\mid_{[\ell-1]}$.
    This means we can add $\ell$ to the cell containing $m_i$ while preserving the standard condition, hence this process terminates with $T \in \ShSet_p(\lambda/\mu)$.

    To upgrade to unimodal factorizations, first observe for $\res(T) = (\bfa,\bfi)$ that $\st(T) \mapsto \bfa$ by the definition of $\res$.
    Therefore $\res: \ShSSYT(\lambda/\mu) \to \UU(w_{\lambda/\mu})$ as desired.
    For the inverse, let $T' = \res^{-1}(\bfa)$.
    Replacing the elements corresponding to $\bfa^k$ with $\overline{k}$, the descending condition guarantees they will be row strict and column weak.
    Likewise, when replacing the elements corresponding to $\bfb^k$ with $k$, the ascending condition guarantees they will be column strict and row weak.
    Therefore, $\res$ is invertible, hence the result follows.
\end{proof}

Our motivation for proving this is the following:

\begin{corollary}
    \label{c:skew-GQ}
    For $\mu \subseteq \lambda$ strict shapes, $G^C_{w_{\lambda/\mu}} = GQ^{(\beta)}_{\lambda/\mu}$.
\end{corollary}

\begin{proof}
    Combine Proposition~\ref{p:unimodal_GC} with Theorem~\ref{t:res} and the definition~\eqref{eq:GQ} of $GQ^{(\beta)}_{\lambda/\mu}$. 
\end{proof}

This result generalizes to products for the form $GQ$ times $GS$, and more generally $G^C$ times $GS$ as follows.
Let $\rho$ be a strict partition with $\ell(\rho) = k$ and $\rho_k > n$.
For $\mu = (n+k-1,n+k-2,\dots,n)$, the $C$ tableau construction shows the minimum entry in a reduced word for $w_{(\lambda/\mu)}$ is $n$ so it fixes $i < n$.
By Corollary~\ref{c:skew-GQ}, we see $G^C_{w(\rho/\mu)} = GS_\nu$ where $\nu$ is the (not strict) partition with $i$th entry $\rho_i +i - n- k$.
Therefore, by Proposition~\ref{p:signed-cross-ordinary} for $w \in W_n$ we have
\begin{equation}
    \label{eq:GC-GS}
G^C_{w\cdot w(\lambda/\mu)} = G^C_w \cdot GS_\nu
\end{equation}
We highlight an important special case:
\begin{corollary}
    \label{c:GQ-GS}
    Let $\lambda$ be a strict partition with $\lambda_1 < a$ and $\nu$ be a (not necessarily) strict partition with $\ell(\nu) = k$.
    Define $\rho$ so that $\nu_i = \rho_i + i - a - k$, and let $\mu = (a+k-1,a+k-2,\dots,a)$.
    Then
    \[
G^C_{w(\lambda) \cdot w(\rho/\mu)} = GQ_{\lambda} \cdot GS_\nu.
    \]
\end{corollary}

\section{Strict Decomposition Tableaux}
\label{ss:strict-tableaux}

In this section, we introduce a notion of strict deomposition tableaux, simultaneously generalizing the reduced decomposition tableaux appearing in~\cite{kraskiewcz1989reduced,lam1995b} and the standard decomposition tableaux appearing in~\cite{serrano2010shifted}\footnote{Confusingly, reduced decomposition tableaux are referred to as `standard' in~\cite{lam1995b}. We introduce the term `reduced' as the Serrano tableaux more closely resemble the conventional meaning of `standard'.}.
In order to define strict decomposition tableau, we require some preliminaries.

\begin{definition}
\label{d:unimodal}
A sequence $R = r_1 r_2 \dots r_k$ of natural numbers is a \defn{(strictly) unimodal} if there exists $j \in [k]$ so that 
\[
r_1 > r_2 > \dots > r_j < \dots < r_k.
\]
Depending on context, we refer to both $j$ and $r_j$ as the \defn{dip} of $R$.
The \defn{decreasing} and \defn{increasing parts} of $R$ are $R^\downarrow = r_1 \dots r_{j-1}$ and $R^\uparrow = r_{j} \dots r_k$, respectively.
Note the increasing part is always non-empty, which is at odds with the conventions in~\cite{lam1995b}.
\end{definition}

For shifted tableau with a unimodal row $R$ in row $i$ with dip $j$, it is frequently convenient to view it as
\begin{align*}
\mathsf{T}(R) &:= -r_i < \dots < -r_{j-1} < r_j < r_{j+1} < \dots < r_k,\ \text{or}\\
\mathsf{B}(R) &:= -r_i < \dots < -r_{j-1} < -r_j < r_{j+1} \dots < r_k,
\end{align*}
which are increasing sequences of integers.

\begin{remark}
Often, unimodal sequences are first increasing, then decreasing.
However, the previous observation demonstrates that working with sequences that first decrease, then increase is more natural in our context.
For simplicity, we use the term unimodal throughout, as opposed to `counimodal'.
\end{remark}

We are now prepared to define the main object of this section.

\begin{definition}
\label{d:strict-decomposition}
    Let $\lambda$ be a strict partition with $\ell(\lambda) = k$.
    A \defn{strict decomposition tableau} is 
    a tableau $T:D_\lambda \rightarrow \mathbb{N}$ with rows $R_1 \dots R_k$ so that:
    \begin{enumerate}[(a)]
        \item For all $i \in [k]$, the row $R_i$ is unimodal.
        \item For $i \in [k-1]$, the first and last entries of $R_{i+1}$ are less than the first entry of $R_i$.
        \item \label{sdt.c} For $i \in [k-1]$, consider the increasing sequences $\mathsf{T}(R_i) = a_i \dots a_{i+\lambda_i}$ and $\mathsf{B}(R_{i+1}) = b_{i+1} \dots b_{i+1+\lambda_{i+1}}$.
        For all $j \in [\lambda_{i+1}]$,
        \begin{equation}
        \label{eq:row-rule}
        \{ \pm a_i, \dots, \pm a_{j-1},\pm b_{j+1},\dots,\pm b_{i+1+\lambda_{i+1}}\} \cap (b_j,a_{j}] = \varnothing.    
        \end{equation}
    \end{enumerate}
    In Condition (c), note that $a_{j}$ appears immediately above $b_j$ in $T$. If $a_{j} \leq b_j$, then we will have that $(b_j,a_{j}] = \varnothing$, so Equation~\eqref{eq:row-rule} is satisfied vacuously.
If $x$ is an element of the set defined in the LHS of Equation~\eqref{eq:row-rule}, we say $x$ \defn{witnesses} $b_j < a_{j+1}$, hence the failure of Condition (\ref{sdt.c}).\\

\begin{comment}
Alternatively we can describe the strict decomposition tableau as avoiding the following configurations between adjacent rows:
\[
\begin{ytableau}
    a & \cdots & \\
    \none & \cdots & b
\end{ytableau},
\begin{ytableau}
    \; & \cdots & a & \cdots & \\
    \none & \cdots & c & \cdots & b
\end{ytableau},
\begin{ytableau}
    \; & \cdots & v & z & \cdots & \\
    \none & \cdots & \cdots & x & \cdots & y
\end{ytableau},
\begin{ytableau}
    y & \cdots & z\\
    \none & \cdots & x
\end{ytableau},
\begin{ytableau}
    \; & \cdots & y & \cdots & z\\
    \none & \cdots & \; & \cdots & x
\end{ytableau}
\]
with $a\leq b < c$, $x < y \leq z$, and $v < z$.
\end{comment}
\end{definition}

\begin{lemma}\label{lem:config}
    A shifted tableau with unimodal rows is a strict decomposition tableau if and only if the tableau avoids the following five configurations:
\[
(i)\;\begin{ytableau}
    a & \cdots & \\
    \none & \cdots & b
\end{ytableau}\,,\ 
(ii)\;\begin{ytableau}
    \; & \cdots & a & \cdots & \\
    \none & \cdots & c & \cdots & b
\end{ytableau}\,,\ 
(iii)\;\begin{ytableau}
    \; & \cdots & v & z & \cdots & \\
    \none & \cdots & \cdots & x & \cdots & y
\end{ytableau}\,,\ 
(iv)\;\begin{ytableau}
    y & \cdots & z\\
    \none & \cdots & x
\end{ytableau}\,,\ 
(v)\;\begin{ytableau}
    \; & \cdots & y & \cdots & z\\
    \none & \cdots & \; & \cdots & x
\end{ytableau}
\]
with $a\leq b < c$, $x < y \leq z$, and $v < z$.
\end{lemma}

\begin{proof}
    Suppose a tableau contains configuration $(i)$, then the tableau breaks rule (b) as $b > a$ where $a$ is the first entry in the row above $b$.\\
    Suppose a tableau contains configuration $(ii)$, then $-c < \pm a$ is witnessed by $b\in (-c, \pm a]$ meaning rule (c) is broken.\\
    Suppose a tableau contains configuration $(iii)$, then as $v < z$ we have that $z$ is in the increasing part of it's row.
    Thus, as $x < z$ we always have $y$ witnesses $z$ over $x$ resulting in the tableau breaking rule $(iii)$.\\
    Suppose a tableau contains configuration $(iv)$ or $(v)$, then as $y \leq z$ we have that $z$ is in the increasing part of it's row.
    Thus, as $x < z$ we always have $y$ witnesses $z$ over $x$ resulting in the tableau breaking rule $(iii)$.\\
    Therefore, we have that if a tableau contains any of the five configurations, then it is not a strict decomposition tableau.

    \noindent Suppose a tableau fails rule (b) at row $i$.
    Let $a$ be the first entry of row $R_{i}$ and $b$ be the entry of row $R_{i+1}$ that is greater or equal to $a$.
    Then $a\leq b$ forms configuration $(i)$.\\
    Suppose a tableau fails rule (c).
    Let $p\in R_{i+1}$, $q\in R_{i}$ such that $q$ above $p$ is witnessed by an element $w$.\\
    \textbf{Case 1:} Both $p$ and $q$ are in the decreasing part of their rows.
    Note that $q \leq w < p$.
    If $w\in R_{i+1}$, then we have that $w$ is right of $p$ and $q < w < p$.
    Therefore, the tableau contains configuration $(ii)$ with $a = q$, $b = w$, and $c = p$.
    If $w\in R_{i}$, then let $t$ be the element directly below $w$.
    If $t$ does not exist then the tableau forms configuration $(i)$ with $a = w$ and $b = p$.
    Otherwise, we have that $t > p$ as $p$ is in the decreasing part of it's row.
    However, then the tableau forms configuration $(ii)$ with $a = w$, $b = p$, and $c = t$.\\
    \textbf{Case 2:} $p$ is in the decreasing part of it's row and $q$ is in the increasing part of it's row.
    Note that $w < q$.
    If $w\in R_{i+1}$, take $t$ to be the element directly above $w$.
    As $q$ is in the increasing part of it's row we have that $q < t$ and as $w$ witnesses $q$ above $p$ we also have $w < q$.
    Therefore, the tableau forms configuration $(v)$ with $x = w$, $y = q$, and $z = t$.
    If $w\in R_{i}$, take $t$ to be the element directly below $w$.
    If $t$ does not exist, then the tableau either forms configuration $(i)$ or $(iv)$ with $a = w$ and $b = p$ if $w \leq p$ and $x = p$, $y = w$, $z = q$ otherwise.
    If $t$ does exist, then we have that $t > p$ as $p$ is in the decreasing part of its row and $t$ is left of $p$.
    Then $w$ above $t$ is witnessed by $p$ bringing us to either case 1 or the earlier part of this case.
    In either situation we have shown that the tableau contains one of the configurations.\\
    \textbf{Case 3:} Both $p$ and $q$ are in the increasing part of their row.
    Note that $p < w \leq q$.
    If $w\in R_{i+1}$, then let $t$ be the element directly left of $q$.
    As $q$ is in the increasing part of it's row we have that $t < q$.
    Therefore, the tableau contains configuration $(iii)$ with $v = t$, $x = p$, $y = w$, and $z = q$.
    If $w\in R_{i}$, then the tableau forms configuration $(iv)$ or $(v)$ with $x = p$, $y = w$, $z = q$.
\end{proof}

If a tableau follows all the conditions of a strict decomposition except with weakly unimodal rows instead or equivalently avoids each configuration in Lemma~\ref{lem:config} with weakly unimodal rows instead we call the tableau a \defn{pseudo strict decomposition tableau}.

\begin{lemma}\label{l:SDT unimodal}
    For $T$ a strict decomposition tableau with rows $R_1,\ldots, R_k$, for each $i\in [k]$, $R_i$ is a unimodal subsequence of maximal length in $R_k R_{k-1}\ldots R_i$.
\end{lemma}

\begin{proof}
    Let $t_1 t_2 \ldots t_{|T|}$ be the reading word of $T$.
    Let $r_1 \ldots r_p$ be a unimodal subsequence of the reading word with $r_1, \ldots, r_{\ell}$ being in row $j$.
    We aim to show that there exists another unimodal subsequence of equal or greater length that starts in row $j-1$.\\
    Suppose $r_{\ell+1}$ is in any row above row $j-1$.
    Let $s$ be the element directly above $r_{\ell}$ and let $s'$ be the element directly left of $s$.
    If $r_{\ell} < r_{\ell+1}$ and if $s < r_{\ell+1}$, then $r_1, \ldots, r_{\ell}$ can be replaced with the elements in row $j-1$ up to $s$.
    If $r_{\ell} < r_{\ell+1}$ and $s \geq r_{\ell+1}$, then in order to avoid configuration $(v)$ we either have $s' > s$ or $s' < r_{\ell}$.
    In both cases we can replace $r_1, \ldots, r_{\ell}$ with row $j-1$ up to $s'$ as in the first case row $j-1$ is decreasing up to $s'$.
    Therefore, we will assume $r_{\ell+1}$ is in row $j-1$ as in the case where it is not we can raise elements until we have this property.\\
    \textbf{Case 1:} Let $r_{\ell+1}$ be in row $j-1$ weakly left of $r_{\ell}$.
    Note that this implies that $r_{\ell} < r_{\ell+1}$.
    This is due to the fact that if the unimodal sequence is decreasing up to $r_{\ell}$ then row $j$ must be decreasing up to $r_{\ell}$ as well.
    Therefore the element below $r_{\ell+1}$ is greater than $r_{\ell}$ if it exists.
    However, then the tableau would form configuration $(i)$ or $(ii)$ with $a = r_{\ell+1}$, $b = r_{\ell}$, and $c$ being the element below $r_{\ell+1}$ if it exists.
    Thus, we have that $r_{\ell} < r_{\ell+1}$.
    Then the element above $r_{\ell}$ must be less than $r_{\ell+1}$, as otherwise that element would be the $z$ in a $(iv)$ or $(v)$ configuration with $x = r_{\ell}$ and $y = r_{\ell+1}$.
    This allows us to then replace $r_1, \ldots, r_{\ell}$ with row $j-1$ up to the element above $r_{\ell}$ and keep the sequence unimodal.\\
    \textbf{Case 2:} Let $r_{\ell+1}$ be in row $j-1$ right of $r_{\ell}$.
    If $r_{\ell} < r_{\ell+1}$, then we can simply replace $r_1,\ldots, r_{\ell}$ with row $j-1$ up to $r_{\ell+1}$.
    If $r_{\ell} > r_{\ell+1} < r_{\ell+2}$ (or $r_{\ell+2}$ does not exist), then we can replace $r_1,\ldots, r_{\ell}$ with row $j-1$ up to $r_{\ell+1}$.
    If $r_{\ell} > r_{\ell+1} > r_{\ell+2}$, let $s$ be the element directly above $r_{\ell}$.
    If $r_{\ell} > s > r_{\ell+2}$, then we can replace $r_{\ell+1}$ with $s$ and then reduce to sequence in row $j-1$ and above by Case 1.
    If $r_{\ell} \leq s$, then we can replace $r_1,\ldots, r_{\ell}$ with row $j-1$ up to $s$.\\
    Thus we have shown given any unimodal subsequence of $T$ that doesn't start in row $1$, there exists a unimodal subsequence of equal or greater length that starts a row higher.
    Therefore, as the first row is unimodal it must be a unimodal subsequence of maximal length.
\end{proof}

The following tableaux were introduced as key objects for Kra\'siewicz insertion.

\begin{definition}
\label{d:standard-decomposition}
    Let $\lambda$ be a strict partition with $\ell(\lambda) = k$.
    A \emph{reduced decomposition tableau} of \emph{shape} $\lambda$ is a tableau $T:D_\lambda \rightarrow \mathbb{N}$ with rows $R_1 \dots R_k$ such that the reading word is a reduced word and $R_i$ is a unimodal subsequence of maximal length in $R_k R_{k-1} \ldots R_{i}$ for all $i\in [k]$.
\end{definition}

\begin{proposition}
    Every reduced decomposition tableau is also a strict decomposition tableau.
\end{proposition}

\begin{proof}
    Let $T$ be a reduced decomposition tableau with rows $R_1, \ldots, R_k$.
    By definition we have each row is unimodal, thus $T$ satisfies $(a)$.
    Suppose the first or last entry of $R_{i+1}$, $a$, is greater than the first entry of $R_i$. 
    Then $aR_i$ would be a unimodal subsequence of $R_k R_{k-1} \ldots R_{i}$, contradicting $T$ being a reduced decomposition tableau.
    Thus, $T$ satisfies $(b)$.

By way of contradiction, assume $T$ breaks condition (c) in rows $R_i$ and $R_{i+1}$ with $\pm a_j$ in $R_i$ greater than $\pm b_{j}$ in $R_{i+1}$ and witness $c$.

    \noindent \textbf{Case 1:} $c \neq a_j$.
    Let $R'_{i+1}$ be $R_{i+1}$ up to position $j$ and $R'_{i}$ be $R_{i}$ without the first $j-1$ entries.
    Then $R'_{i+1} c R'_{i}$ is a unimodal subsequence of $R_k R_{k-1}\ldots R_{i}$ of length greater than $R_{i}$, a contradiction.\\

    \noindent \textbf{Case 2:} Suppose that there is no witness that satisfies the conditions of case 1, that is every witness that exists is equal to the top element.
    Take $c$ to be the greatest such witness.
    
    \noindent \textbf{Case 2.1:} Suppose $c=0$.
    Then we must have that $a_j = 0$, $c$ must be in the bottom row at position $j$, and there must be a $1$ between $a_j$ and $c$.
    However, that $1$ will always be witnessed by $0 = b_{j+1} < a_{j+1}$ contradicting our assumption that $c$ was maximal.
    
    \noindent \textbf{Case 2.2:} Suppose $a_j$ is in the increasing part of $R_i$.
    
    \noindent \textbf{Case 2.2.1:} Suppose $c=1$, which requires then that $b_{j} = 0$.
    In order for the reading word to be reduced there must be a $2$ or $0$ between $a_j$ and $c$.
    If $c$ is in $R_i$, then there cannot be $c+1$ between $a_j$ and $c$.
    If $c$ is in $R_{i+1}$, then as $a_{j+1} > 1$ any $2$ between $a_j$ and $c$ would witness $1 = b_{j+1} < a_{j+1}$ contradicting $c$ maximal.
    Thus, there must be a $0$ between $a_j$ and $c$, specifically $a_{j-1}$ must be $0$.
    Further, there cannot be any $2$ between $b_{j-1}$ and $c$ as if $c$ is in $R_i$ then the element below $c$ is at least $2$ which then contradicts that $c$ is maximal and if $c$ is in $R_{i+1}$ such a $2$ would mean $R_{i+1}$ is not unimodal.
    However, then the tableau has a $10101$ pattern which is not reduced.
    
    \noindent \textbf{Case 2.2.2:} Suppose $c > 1$.
    In order for the reading word to be reduced there must be a $c+1$ or $c-1$ between $a_{j}$ and $c$.
    If $c$ is in $R_i$ then there cannot be $c+1$ between $a_j$ and $c$ without making $R_i$ not unimodal.
    If $c$ is in $R_{i+1}$ then the element above $c$ is greater than $a_j$ meaning that a $c+1$ between $c$ and $a_j$ would be witnessed contradicting $c$ being maximal.
    Therefore, there is a $c-1$ between $a_j$ and $c$.
    Further, in order to not be in case $1$ we require that $b_{j}=c-1$ otherwise $b_{j} < a_j$ witnesses $c-1$.
    However, the reading word then has a $c-1\; c\; c-1\; c$ pattern with $c > 1$, meaning the word is not reduced.

    \noindent \textbf{Case 2.3:} Suppose $a_j$ is in the decreasing part of $R_i$, note that this means $c$ is in $R_{i+1}$.
    In order for the reading word to not be reduced there must be a $c+1$ or $c-1$ between $a_j$ and $c$.

    \noindent \textbf{Case 2.3.1:} Suppose $c-1$ is between $a_j$ and $c$.
    If $c$ is in the increasing part of $R_{i+1}$ then since $R_i$ and $R_{i+1}$ are unimodal there cannot be and $c-1$ between $a_j$ and $c$.
    Further, if $c$ is in the decreasing part of $R_{i+1}$ then the element above $c$ witnesses $c-1$.
    In order to not be case 1 we then require that the element above $c$ is also $c-1$.
    However, then the reading word will not be reduced unless $c=1$ in which case the element above $c-1$ has witnesses such that we are either in case 1 or case 2.2.1.

    \noindent \textbf{Case 2.3.2:} Suppose $c+1$ is between $a_j$ and $c$, which requires that $b_{j}=c+1$ to not be in case 1.
    Further, suppose $c-1$ is not between $a_j$ and $c$.
    Then the reading word has a $c c+1 c c+1$ pattern meaning that the reading word is not reduced unless $c=0$.
    However, then the element above $c$ witnesses the $c+1$ between $a_j$ and $c$ meaning this case is either covered in case 1 or case 2.2.1.

\end{proof}

\section{Kra\'skiewicz--Hecke Insertion}
\label{s:insertion}

The goal of this section is to construct a bijection from words in the alphabet $\mathbb{N} = \{0,1,2,\dots\}$ to pairs of tableaux $P,Q$ of the same shape where $P$ is a \emph{strict decomposition tableau} and $Q$ is a \emph{standard set-valued shifted tableau}.
This algorithm generalizes Kra\'skiewicz insertion, introduced in~\cite{kraskiewcz1989reduced} and further studied in~\cite{lam1995b}, which maps reduced words for signed permutations to pairs $P,Q$ where $P$ belongs to a subclass of strict decomposition tableau and $Q$ is standard shifted but not set-valued.

\subsection{Insertion}
\label{ss:insertion}

We are now prepared to introduce our insertion algorithm, beginning with the row insertion rule.
This algorithm is very similar to Kra\'skiewicz row insertion, but requires a plethora of additional cases to correct situations where the resulting tableau would fail to be a strict decomposition tableau.

\begin{definition}
    \label{d:row-insertion}
    \emph{(Kra\'skiewicz--Hecke) row insertion} is an algorithm with inputs $a \in \mathbb{N}$ and a two-row strict decomposition tableau $RS$ and outputs $b \in \mathbb{N} \cup \{\infty\}$ and two-row strict decomposition tableau $R'S$.
    Note $S$ is unchanged and that $S$ may be empty.
    Row insertion is a two step procedure, first applying \emph{right insertion} and then \emph{left insertion}, which in some cases is trivial.
    Let $R = r_1 \dots r_\ell$ with dip $r_q$ and $S = s_2 \dots s_k$ with dip $s_p$ and $k \leq \ell$.
    Set $r_0, s_1 = -\infty$, $r_{\ell+1} = \infty$ and $s_m = \infty$ for $m > k$.
    Note $r_i$ appears immediately above $s_i$ in $RS$. For any $d\in \N$, define $A_d = \{s_{d+1},\dots,s_k,r_1,\dots,r_{d-1}\}$.
    
    We first define right insertion, which outputs $a' \in \mathbb{N} \cup \{\infty\}$ and row $R''$, which is an input for left insertion.
    If $a$ equals the dip $r_q$, set $i=q$.
    Otherwise, let $q < i \leq \ell+1$ be minimal such that $r_i \geq a$.
    \begin{enumerate}[(R1)]
        \item If $a \neq r_i$, then set $a' = r_i$ and create $R''$ by changing $r_i$ to  $a$.
        \item If $a = r_i$ and $r_{i-1} > r_{i+1}$, then set $a' = r_{i+1}$ and create $R''$ by setting $r_{i+1} = r_i$, note the resulting tableau will be a pseudo SDT.
        \item If $a = r_i$ and $r_{i-1} \leq r_{i+1}$, then let $R''=R$ and set $a' = \min ( \{r_{i+1}\} \cup [ (s_i, r_{i+1}) \cap A_i])$.
    \end{enumerate}

    We now define left insertion. If $a'=\infty$, then set $b=\infty$ and $R'=R''$. Otherwise, let $1\leq j \leq q$ be minimal such that $r_j \leq a'$ where here $r_j$ indicates the $j$\textsuperscript{th} entry of $R''$.
    \begin{enumerate}[(L1)]
        \item If $a' \neq r_j$, set $b = r_j$ and create $R'$ by changing $r_j$ to $a'$.
        \item If $a' = r_j$, then set $R' = R''$.
        \begin{enumerate}[(I)]
            \item If $j < p$, we define 
            \[
            b = \begin{cases}
            \max(A_j \cap (r_{j+1}, s_j)), & \text{if } j+1 < q \text{ and } A_j \cap (r_{j+1},s_j) \neq \emptyset\\
            \max(A_j\cap (-\infty, s_j)), & \text{if } j+1 \geq q \text{ and } A_j \cap (-\infty,s_j) \neq \emptyset\\
            r_{j+1}, & \text{otherwise.}
            \end{cases}
            \]
            
            \item If $j \geq p$, then define
            \[
            b = \begin{cases}
            s_{j+1}, & \text{ if } (a)\; j+1 \geq q\ \text{and}\  (b)\; r_{j+1} > s_{j+1} \text{ or } r_{j+2}, s_{j+2} > r_j > s_{j+1}\\
            r_{j+1}, & \text{otherwise.}
            \end{cases}
            \]
        \end{enumerate}
\end{enumerate}
    
\end{definition}

\begin{example}
We present examples of the various steps of the insertion.

    \begin{enumerate}
\item Here $r_i = \infty$ triggering (R1), followed by left inserting $\infty$ triggering (L1) which ends the process:
    \[
        \centering
    \begin{ytableau}\none&3&0&\none[\;\;\;\leftarrow2]\\\end{ytableau}
\quad \quad \quad
    \begin{ytableau}\none[\infty\rightarrow\;\;\;]&3&0&2\\\end{ytableau}
\quad 
\begin{ytableau}
        \none & 3 & 0 & 2\\
    \end{ytableau}
    \]\vspace{.25cm}

\item Here $r_i = 2$ triggering (R1), then $r_j = 0 \neq 2$ triggering (L1):
\[
\centering
    \begin{ytableau}\none&3&0&2&\none[\;\;\;\leftarrow1]\\\end{ytableau}
\quad \quad \quad
    \begin{ytableau}\none[2\rightarrow\;\;\;]&3&0&1&\none\\\end{ytableau}
\begin{ytableau}\none&3&2&1&\none\\
    \none&\none& &\none[\;\;\;\leftarrow0]\\\end{ytableau}
\]\vspace{.25cm}
\item Here $r_i = 0$ triggering (R2) as $2 > 1$, then left insertion of $0$ triggers (L1). When the bottom row is empty we take $p=0$, thus we get that $0$ is inserted into the next row:
\[
\begin{ytableau}\none&2&0&1&\none[\;\;\;\leftarrow0]\\\end{ytableau}
\quad \quad \quad
\begin{ytableau}\none[1\rightarrow\;\;\;]&2&0&0&\none\\\end{ytableau}
\quad \quad \quad
\begin{ytableau}\none&2&1&0&\none\\
\none & \none & & \none[\;\;\;\leftarrow0]\\
\end{ytableau}
\quad \quad \quad
\begin{ytableau}\none&2&1&0&\none\\
\none & \none & 0 & \none\\
\end{ytableau}
\]
\vspace{.25cm}

\item Here $r_i = 1$ triggering (R3) as $0 \leq \infty$. As $r_{i+1}$ and $s_i$ are $\infty$ we get $a'=r_{i+1}=\infty$, terminating as in (1):
\[
\begin{ytableau}\none&2&0&1&\none[\;\;\;\leftarrow1]\\\end{ytableau}
\quad \quad \quad
    \begin{ytableau}\none[\infty\rightarrow\;\;\;]&2&0&1&\none\\\end{ytableau}
\]
\vspace{.25cm}

\item Here $r_i = 1$ triggering (R3) as $0 \leq 3$. As $2 \in (1,3)$ and there is a $2$ between $s_i$ and $r_i$, we get $a'=2$:
\[
\begin{ytableau}\none&*(green)4&*(green)2&*(green)0&1&*(green)3&\none[\;\;\;\leftarrow1]\\\none&\none&2&0&1\\\end{ytableau}
\quad \quad \quad
    \begin{ytableau}\none[2\rightarrow\;\;\;]&4&2&0&1&3&\none\\\none&\none&2&0&1\\\end{ytableau}
\]\vspace{.25cm}

\item Here $r_j = 1$ triggering (L2). As $j = 2 < p = 3$ and $j+1 = 3 \not< q = 3$, we set $b$ to $\max (-\infty, 2) \cap \{0,1,4\}$ which is $1$:
\[
\begin{ytableau}\none[2\rightarrow\;\;\;]&*(green)4&2&0&1&3&\none\\\none&\none&2&*(green)0&*(green)1\\\end{ytableau}
\quad \quad \quad
    \begin{ytableau}\none&4&2&0&1&3&\none\\\none&\none&2&0&1&\none[\;\;\;\leftarrow1]\\\end{ytableau}
\]\vspace{.25cm}

\item Here $r_j = 3$, triggering (L2). As $j = 2 < p = 3$ and $j+1 = 3 < q = 4$, we set $b$ to $(1, 3) \cap \{2, 4\}$ which is $2$:
\[
\begin{ytableau}\none[3\rightarrow\;\;\;]&*(green)4&3&1&0&1&3\\\none&\none&3&*(green)2\\\end{ytableau}
\quad \quad \quad
    \begin{ytableau}\none&4&3&1&0&1&3\\\none&\none&3&2&\none[\;\;\;\leftarrow2]\\\end{ytableau}
\]\vspace{.25cm}

\item Here $r_j = 4$, triggering (L2). As $j = 2 \not< p = 2$, we are in case (II). (a) and (b) hold as $j + 1 = 3 \geq q = 3$, $j = 2 \geq p = 2$, and $r_{j+1} = 2 > s_{j+1} = 1$. Thus, we set $b = s_{j+1} = 1$:
\[
\begin{ytableau}\none[4\rightarrow\;\;\;]&5&4&2&3\\\none&\none&0&1\\\end{ytableau}
\quad \quad \quad
    \begin{ytableau}\none&5&4&2&3\\\none&\none&0&1&\none[\;\;\;\leftarrow1]\\\end{ytableau}
\]\

\vspace{.25cm}
\item Here $r_j = 2$ triggering (L2). As $j = 2 \not< p = 2$, we are in case (II). (a) and (b) hold as $j + 1 = 3 \geq q = 3$, $j = 2 \geq p = 2$, and $r_{j+2}, s_{j+2} = 3,3 > r_j = 2 > s_{j+1} = 1$. Thus, we set $b = s_{j+1} = 1$:
\[
\begin{ytableau}\none[2\rightarrow\;\;\;]&4&2&1&3&\none\\\none&\none&0&1&3\\\end{ytableau}
\quad \quad \quad
    \begin{ytableau}\none&4&2&1&3&\none\\\none&\none&0&1&3&\none[\;\;\;\leftarrow1]\\\end{ytableau}
\]

\vspace{.25cm}

\item Here $r_j = 2$ triggering (L2). As $j = 2 \not< p = 2$, we are in case (II). (b) does not hold as $r_{j+1} = 1 \not> s_{j+1} = 2$ and $r_{j+2} = 0 \not> r_j = 2$. Thus, $b = r_{j+1} = 1$:
\[
\begin{ytableau}\none[2\rightarrow\;\;\;]&4&2&1&0\\\none&\none&2&3\\\end{ytableau}
\quad \quad \quad    
    \begin{ytableau}\none&4&2&1&0\\\none&\none&2&3&\none[\;\;\;\leftarrow1]\\\end{ytableau}
\]

\end{enumerate}
\end{example}

We extend row insertion to an insertion algorithm by repeated application.

\begin{definition}
    \label{d:full-insertion}
For $P = R_1 \dots R_\ell$ a strict decomposition tableau and $a\in \mathbb{N}$, we \emph{Kra\'skiewicz--Hecke insert} $a$ into $P$ by row inserting $a$ into $R_1R_2$, updating $R_1$ and inserting the output $b$ into $R_2R_3$ and so on until the output is $\infty$.
The insertion \emph{terminates in row $i$} where $i$ is the row whose output from row insertion is $\infty$.
    
    For $\bfa = (a_1,\dots,a_p) \in \mathbb{N}^p$ and $\bfa'=(a_1,\dots,a_{p-1})$, we define the \emph{Hecke--Kra\'skiewicz insertion tableau} $P_{HK}(\bfa)$ recursively by row inserting $a_p$ into $P_{HK}(\bfa')$.
    The \emph{Hecke--Kra\'skiewicz recording tableau} $Q = Q_{HK}(\bfa)$ is also constructed recursively from $Q' = Q_{HK}(\bfa')$.
    Let $\lambda$ and $\lambda'$ be the shapes of $P_{HK}(\bfa)$ and $P_{HK}(\bfa')$, respectively.
\begin{enumerate}
    \item If $\lambda \neq \lambda'$, they differ by a single cell $(i,j)$.
    We obtain $Q$ from $Q'$ by setting $Q_{ij} = \{p\}$.
    \item If $\lambda = \lambda'$, let $k$ be the row where row insertion of $a_p$ into $P_{HK}(\bfa')$ terminates.
    Let $\ell = \max \{\ell: \lambda_\ell + \ell - k = \lambda_k\}$, and construct $Q$ from $Q'$ by adding $p$ to $Q'_{\ell\lambda_\ell}$.

\end{enumerate}
    
\end{definition}

We show this algorithm is well-defined in \Cref{p:row-insertion}.
Many of our arguments depend on understanding the sequences of left and right positions output be the insertion.

\begin{definition}
    Consider the KH--row insertion of $a \in \mathbb{N}$ into the $i$\textsuperscript{th} row of the strict decomposition tableau $P$ with $a$ bumping $a'$, which is then left inserted to obtain $P'$.
    The \emph{right position} is the largest $c$ so that $P'_{ic} = a$.
    Likewise, the \emph{left position} is the smallest $d$ so that $P'_{id} = a'$.
    When inserting $a\in \mathbb{N}$ into a strict decomposition tableau $P$, we obtain a sequence of $k$ row insertions.
    The \emph{right bumping path} is the sequence $(c_1,\dots,c_k)$ whose $i$\textsuperscript{th} entry is the right position of the $i$\textsuperscript{th} row insertion.
    Similarly, the \emph{left bumping path} is the sequence $(d_1,\dots,d_{k-1})$ whose $i$\textsuperscript{th} entry is the left position of the $i$\textsuperscript{th} row insertion.
\end{definition}
We present an example of a complete Kra\'skiewicz--Hecke insertion with bumping path.

\begin{example}
\label{ex:bumping-paths}
We present an example of $KH$ insertion into a relatively large tableau.
The right bumping path of the insertion is colored in red while the left bumping path is colored in green.
Note both bumping paths move weakly left as required by Lemma~\ref{l:bumpingpath}.
The insertion rule applied is to the right of the tableau.
\[\ytableausetup{smalltableaux}
    \begin{tabular}{cc}
\begin{ytableau}\none & 8 & 7 & 6 & 4 & 3 & 1 & 0 & 1 & 3 & 4 & 7 & \none& \none[\;\leftarrow 0]\\ \none &\none & 7 & 6 & 4 & 2 & 1 & 0 & 3 & 4 & 7 & \none\\ \none &\none &\none & 6 & 4 & 1 & 3 & \none\\ \none &\none &\none &\none & 5 & \none\\ \none &\none &\none &\none &\none & \none\\ \end{ytableau}
\;\;\;\;\;(R3)\textcolor{white}{(II)}
&
\begin{ytableau}\none[1\rightarrow\;] & \none & 8 & 7 & 6 & 4 & 3 & 1 & *(red) 0 & 1 & 3 & 4 & 7 & \none\\ \none & \none &\none & 7 & 6 & 4 & 2 & 1 & 0 & 3 & 4 & 7 & \none\\ \none & \none &\none &\none & 6 & 4 & 1 & 3 & \none\\ \none & \none &\none &\none &\none & 5 & \none\\ \none & \none &\none &\none &\none &\none & \none\\ \end{ytableau}
\;\;\;\;\;(L2)(I)\textcolor{white}{I}
\\
\begin{ytableau}\none & 8 & 7 & 6 & 4 & 3 & *(green) 1 & *(red) 0 & 1 & 3 & 4 & 7 & \none\\ \none &\none & 7 & 6 & 4 & 2 & 1 & 0 & 3 & 4 & 7 & \none & \none[\;\leftarrow 0]\\ \none &\none &\none & 6 & 4 & 1 & 3 & \none\\ \none &\none &\none &\none & 5 & \none\\ \none &\none &\none &\none &\none & \none\\ \end{ytableau}
\;\;\;\;\;(R3)\textcolor{white}{(II)}
&
\begin{ytableau}\none  & 8 & 7 & 6 & 4 & 3 & *(green) 1 & *(red) 0 & 1 & 3 & 4 & 7 & \none\\ \none[3\rightarrow\;]&\none & 7 & 6 & 4 & 2 & 1 & *(red) 0 & 3 & 4 & 7 & \none\\ \none &\none &\none & 6 & 4 & 1 & 3 & \none\\ \none &\none &\none &\none & 5 & \none\\ \none &\none &\none &\none &\none & \none\\ \end{ytableau}
\;\;\;\;\;(L1)\textcolor{white}{(II)}
\\
\begin{ytableau}\none & 8 & 7 & 6 & 4 & 3 & *(green) 1 & *(red) 0 & 1 & 3 & 4 & 7 & \none\\ \none &\none & 7 & 6 & 4 & *(green) 3 & 1 & *(red) 0 & 3 & 4 & 7 & \none\\ \none &\none &\none & 6 & 4 & 1 & 3 & \none & \none[\;\leftarrow 2]\\ \none &\none &\none &\none & 5 & \none\\ \none &\none &\none &\none &\none & \none\\ \end{ytableau}
\;\;\;\;\;(R1)\textcolor{white}{(II)}
&
\begin{ytableau}\none & 8 & 7 & 6 & 4 & 3 & *(green) 1 & *(red) 0 & 1 & 3 & 4 & 7 & \none\\ \none & \none & 7 & 6 & 4 & *(green) 3 & 1 & *(red) 0 & 3 & 4 & 7 & \none\\ \none &\none[3\rightarrow\;] &\none & 6 & 4 & 1 & *(red) 2 & \none\\ \none &\none &\none &\none & 5 & \none\\ \none &\none &\none &\none &\none & \none\\ \end{ytableau}
\;\;\;\;\;(L1)\textcolor{white}{(II)}
\\
\begin{ytableau}\none & 8 & 7 & 6 & 4 & 3 & *(green) 1 & *(red) 0 & 1 & 3 & 4 & 7 & \none\\ \none &\none & 7 & 6 & 4 & *(green) 3 & 1 & *(red) 0 & 3 & 4 & 7 & \none\\ \none &\none &\none & 6 & 4 & *(green) 3 & *(red) 2 & \none\\ \none &\none &\none &\none & 5 &  \none & \none[\;\leftarrow 1]\\ \none &\none &\none &\none &\none & \none\\ \end{ytableau}
\;\;\;\;\;(R1)\textcolor{white}{(II)}
&
\begin{ytableau}\none & 8 & 7 & 6 & 4 & 3 & *(green) 1 & *(red) 0 & 1 & 3 & 4 & 7 & \none\\ \none &\none & 7 & 6 & 4 & *(green) 3 & 1 & *(red) 0 & 3 & 4 & 7 & \none\\ \none &\none &\none & 6 & 4 & *(green) 3 & *(red) 2 & \none\\ \none &\none &\none &\none & 5 & *(red) 1 & \none\\ \none &\none &\none &\none &\none & \none\\ \end{ytableau}
\;\;\;\;\;\textcolor{white}{(R3)(II)}
    \end{tabular}
\]
    
\end{example}

\begin{lemma}
    \label{l:bumpingpath}
    The right bumping path and left bumping path are weakly decreasing sequences.
\end{lemma}

\begin{proof}
    Consider right inserting $a$ into row $R$ of a SDT $P$ and assume the insertion does not terminate at this row. Let the row below $R$ be $S$, $a'$ the element left inserted into $R$, $b$ the element right inserted to $S$, and $b'$ the element left inserted to $S$.

    \textbf{Right bumping path:} Let $y$ be the element in $S$ below $a$ after $a$ is right inserted into $R$. If $|b| \leq |y|$, then we can see the right bumping path will be weakly decreasing. Now we assume $|b| > |y|$ for the sake of contradiction.
    By the definition of left insertion we have that $|a'| > |b|$.
    Since $|y| < |b| < |a'|$ and $a' \in R^\uparrow$, we see $y < a'$.
    % Also, since $b$ is right inserted into $S$, either $a'$ or $b$ is in $R^\downarrow$.
    Since $b$ is the output of the left insertion of $a'$ into $R^\downarrow$, either $a'$ or $b$ is in $R^\downarrow$. Next we discuss the two cases and find contradiction in each of them.

    \noindent\textbf{Case 1:} Suppose the right insertion of $a$ into $R$ triggers $(R1)$ or $(R2)$, so $a'$ is above $y$ in $P$. Then the presence of $a'$ or $b$ in $R^\downarrow$ witnesses $y < a'$, so $P$ is not a SDT and we get a contradiction.

    \noindent\textbf{Case 2:} Suppose the right insertion of $a$ into $R$ triggers $(R3)$.
    Let $z$ be the (possibly infinite) value after $a$ in $R^\uparrow$.
    By the definition of $(R3)$ we have $|a'| \leq |z|$. Also $|a'| > |y|$, so $|a'| \in (|y|, |z|]$.

    \textbf{Case 2.1:} Suppose $b\in R^{\downarrow}$.
    Since $|y| < |b| < |a'|$, we have $|b| \in (|y|,|q|)$, by the definition of $(R3)$ would give $b = a'$, a contradiction.

    \textbf{Case 2.2:} Suppose $a'\in R^{\downarrow}$ and $b\not\in R^{\downarrow}$, then we know the left insertion of $a'$ into $R^{\downarrow}$ will trigger $(L3)$.
    By the definition of $(R3)$ every element between $a' \in R^\downarrow$ and $a \in R^\uparrow$ must be less than or equal to $|y|$, as otherwise a different element would be left inserted.
    Combining this observation with the fact that $|y| < |b| < |a'|$, and the definition of $(L3)$ shows that $b$ is strictly to the right of $a'$ in $S$.
    Suppose $b\in S^{\uparrow}$, then the element above $b$ must be greater than or equal to $z$ (defined at the start of this case).
    However, as $|b| < |a'| \leq |z|$ we would have had $a'$ witnessing $b$ less than the element above it before insertion, so the tableau wouldn't be a SDT.
    Therefore, $b \in S^\downarrow$.
    Let $x$ be the entry above $b$ in $R$.
    Since $b$ is to the right of $a'$ and to the left of $y$, we see $x$ is between $a' \in R^\downarrow$ and $a \in R^\uparrow$.
    Therefore $|x| \leq |y| < |b|$ so $y$ witnesses $b < x$ in $P$,  a contradiction.
    
    We conclude that the right bumping path of our insertion is weakly decreasing.

    \textbf{Left bumping path:} The argument  is similar, but not identical to the previous case.
    If $b' = \infty$, which is the $(L1)$ case, then the result is vacuously true.
    Otherwise the left bumping path will be weakly decreasing if $|b'| \geq |x|$ or if $x\in S^{\uparrow}$, where $x$ is the element below $a'$ after $a'$ was left inserted into $R$.
    For the sake of contradiction, we suppose that $|b'| < |x|$ with $x\in S^\downarrow$.
    By the definition of right insertion we have $|b| < |b'|$, so $|b| < |x|$ as well. Next we discuss the two cases and find contradiction in each of them.
\begin{comment}
    \noindent\textbf{Case 1:}
    Suppose $b$ was bumped by $a'$ using $(L2)$, then for $b'$ to be left inserted into $S$ we require that either $b$ or $b'$ is in $S^\uparrow$.
    If $b$ or $b'$ is to the left of $x$, we are done.\Jianping{This sentence is a bit strange. What have we done here?}
    \Josh{If $b$ or $b'$ are left of $x$, then $x\in S^\uparrow$ (as the $b/b'$ in question is in $S^\uparrow$). However, this then means that $b'$ will end up somewhere left of $a'$ meaning that bumping path is decreasing.}
    \Jianping{I got it now!}
    Otherwise, since $|b| < |x|$ we have $x < b$ in $P$ witnessed by $b$ or $b'$, a contradiction.
    \Jianping{Suggested rephrase:\\
    \noindent\textbf{Case 1:}
    Suppose $b$ was bumped by $a'$ using $(L2)$. Since $b'$ is left inserted into $S$ we must have either $b$ or $b'$ in $S^\uparrow$.
    If $b$ or $b'$ is to the left of $x$, then $x\in S^{\downarrow}$ as $|b|<|b'|<|x|$ which contradicts our assumption that $x \in S^{\uparrow}$.
    If $b$ or $b'$ is to the right of $x$, then we have $x < b$ in $P$ witnessed by $b$ or $b'$, which is a contradiction.}
\end{comment}

    \noindent\textbf{Case 1:}
    Suppose $b$ was bumped by $a'$ using $(L1)$. Since $b'$ is left inserted into $S$ we must have either $b$ or $b'$ in $S^\uparrow$.
    If $b$ or $b'$ is to the left of $x$, then $x\in S^{\downarrow}$ as $|b|<|b'|<|x|$ which contradicts our assumption that $x \in S^{\uparrow}$.
    If $b$ or $b'$ is to the right of $x$, then we have $x < b$ in $P$ witnessed by $b$ or $b'$, which is a contradiction.
    \textbf{Case 2:}
    Suppose $a'$ is left inserted using $(L2)$, hence $a' \in R^\downarrow$.
    If $b' \in S$, since  $|b'|\in (|b|,|x|)$ the insertion of $a'$ by $(L2)$ would bump $b'$ instead of $b$, a contradiction.
    Necessarily, when $b$ is right inserted to $S$ it triggers $(R3)$, outputting $b'$. Therefore we $b\in S^{\uparrow}$ and $b'$ is in the row below $S$ in $P$.
    
    Let $z$ be the element in $S$ above $b'$, and note $z$ is to the right of $b$ by the definition of $(R3)$, hence also to the right of $x$.
    Since $x\in S^\downarrow$, we must be in case (I) of $(L2)$. Suppose $|z| < |x|$, then we have $|r_{j+1}| < |b| < |z| < |x|$ and $A \neq \emptyset$, which will cause $(L3)$ to output a $|z|$ instead of $|b|$. Therefore we must have$|z| \geq |x|$.

    However now we have $b' < z$ witnessed by $x$, hence $P$ is not a SDT. This is a contradiction, which completes our proof.
\end{proof}

\begin{lemma}\label{lem:insertion_size}
    Right insertion outputs a greater number than the input and left insertion outputs a smaller number than the input.
\end{lemma}

\begin{proof}
    The result is clearly true for (R1), (R2), and (L1).
    Let $R = r_1, \ldots r_{\ell}$ be the row we are inserting into and $S = s_2 \ldots s_k$ the row below $R$.
    We first consider (R3).
    Let $a$ be the element right inserted and $b$ be the output from this right insertion.
    By the definition of (R3) $a$ must already in be in the increasing part of the row, say $r_i = a$ in $R^\uparrow$.
    Further, we have $b \in (s_i, r_{i+1})$ or $b = r_{i+1}$, which is greater than $r_i = a$ as $r_i$ is in $R^\uparrow$.
    Therefore if $a\geq b$ then $b \in (s_i,r_{i+1})$ so $r_i = a \geq b > s_i$, hence $b$ witnesses $r_i$ over $s_i$.
    Therefore, we have $a < b$.\\

    \noindent Next we consider the case (L2).
    Let $c$ be the element left inserted and $d$ be the output from this left insertion.
    By definition of (L2) $c$ must already in be in the decreasing part of the row, say $r_j = c$ in $R^\downarrow$.
    Further, by our last case we know that $c$ cannot be the dip of $R$.
    In case (II) we either have $d = r_{j+1} < r_j = c$ or we have $d = s_{j+1}$.
    In (b) of the latter case, either $d = s_{j+1} < r_{j+1} < r_{j} = c$ or $d = s_{j+1} < r_j = c$.
    Now for case (I) suppose that $c \leq d$.
    If $s_j \in S^{\downarrow}$, case (I), then we have $d < s_j$ or $d = r_{j+1}$.
    The latter contradicts $c \leq d$ as $c$ is not the dip, so $d < s_j$.
    However, if $d < s_j$ then $r_j = a \leq b< s_j$ meaning $b$ witnesses $r_j$ over $s_j$.
    Therefore, we have that $c > d$.
\end{proof}

\begin{corollary}\label{cor:no_dip}
    Right insertion can never output the dip.
\end{corollary}
\begin{proof}
    Right insertion always ends with the input being in the increasing part of the row, thus as the output is greater than the input we cannot output the dip.
\end{proof}

We are now prepared to show that row insertion is well-defined.
\begin{proposition}
    \label{p:row-insertion}
    Given an $a\in\mathbb{N}$ and $T$ a strict decomposition tableau, the tableau $T' = T\leftarrow a$ is also a strict decomposition tableau.
\end{proposition}

\begin{proof}
    By Lemma~\ref{l:bumpingpath}, we have that row insertion always outputs a tableau of a shifted shape as the only way for the tableau to not have a shifted shape is to have right insertion terminate to make the number of boxes in row $R_k$ equal to the number of boxes in row $R_{k-1}$.
    However, in that case the right bumping path would go from a value less than or equal to $k-1 + \lambda_{k-1}$ to $k + \lambda_{k-1}$ contradicting that the right bumping path is a weakly decreasing sequence.

    By Lemma~\ref{lem:config}, a strict decomposition tableau has unimodal rows and avoids the five following configurations:
    \[\ytableausetup{nosmalltableaux}
(i)\;\begin{ytableau}
    a & \cdots & \\
    \none & \cdots & b
\end{ytableau},
(ii)\;\begin{ytableau}
    \; & \cdots & a & \cdots & \\
    \none & \cdots & c & \cdots & b
\end{ytableau},
(iii)\;\begin{ytableau}
    \; & \cdots & v & z & \cdots & \\
    \none & \cdots & \cdots & x & \cdots & y
\end{ytableau},
(iv)\;\begin{ytableau}
    y & \cdots & z\\
    \none & \cdots & x
\end{ytableau},
(v)\;\begin{ytableau}
    \; & \cdots & y & \cdots & z\\
    \none & \cdots & \; & \cdots & x
\end{ytableau}
\]
    with $a\leq b < c$, $x < y \leq z$, and $v < z$.
In this argument, we can treat (iv) and (v) together.
    
    Each insertion step preserves unimodality of rows with the exception of (R2).
    However, (R2) is always followed by (L1) meaning that unimodality of the rows will be regained after the next insertion step.
    Therefore, the result will follow if we show an insertion step resulting in the tableau $T'$ containing a bad configuration can only occur when $T$ is not a (pseudo) strict decomposition tableau.
    We will denote each case as a triple: the configuration found in $T'$, the element in the configuration that was inserted, and $\textbf{R}$ or $\textbf{L}$ depending on whether the insertion was right insertion or left insertion.
    For each case let $t$ be the element that was bumped out by the whatever created the bad configuration.
    Further, let $R = r_1 \dots r_\ell$ be the top of row of the configuration and $S = s_2 \dots s_k$ the bottom row.
    Note if $t$ was bumped by $u$ via right insertion then $t > u$ and if via left insertion $t < u$.\\

    \noindent\textbf{Case (i, a, R)}: This is only possible if $R$ is increasing.
    Then $r_1$ witnesses $r_2$ over $s_2$, so this is impossible.\\

    \noindent\textbf{Case (i, a, L)}: As $t < a \leq b$ we have that $t$ and $b$ form configuration ($i$).\\

    \noindent\textbf{Case (i, b, R)}: As $a \leq b < t$ we have that $a$ and $t$ form configuration ($i$).\\

    \noindent\textbf{Case (i, b, L)}: For $b$ to be left inserted into $S$, either $b$ was already in $S$ or it was in the row below.
    In the first case we already have configuration ($i$) with $a$ and $b$.
    In the second case, in order to avoid configuration ($i$) the first element $s_2$ of $S$ must be greater than $b$.
    However, that would mean $a \leq b < s_2$ forms configuration ($i$).\\

    \noindent\textbf{Case (ii, a, R)}: As $t$ is bumped via right insertion we have $t \in R^{\uparrow}$, and as $c > b$ we have that $c\in S^{\downarrow}$.
    Therefore as $b < c$, $b$ witnesses $t$ over $c$.\\

    \noindent\textbf{Case (ii, a, L)}: As $t < a \leq b < c$ we have that $t$, $b$, and $c$ form configuration ($ii$).\\

    \noindent\textbf{Case (ii, b, R)}: If $t < c$, then $a$, $t$, and $c$ form configuration ($ii$), as $a \leq b < t$.
    Therefore, assume $t \geq c$.
    Then $c\in S^{\downarrow}$ as otherwise $b$ would have bumped $c$.
    In order for $b$ to be right inserted into $S$ we must have had $b$ in $R$ or $S$ before the last insertion step.
    If $b$ was in $R$, we have $b \in R^\downarrow$ as it is the output of left insertion.
    Then $b$ was to the left of $a$, so it would witness $a$ over $c$ since $c \in S^\downarrow$ and $a < b < c$.
    Therefore, we have $b\in S$.
    As $c\in S^{\downarrow}$ we must have that $b$ is right of $c$, so $a$, $b$, and $c$ already formed configuration ($ii$).\\

    \noindent\textbf{Case (ii, b, L)}: If $t\geq a$, we see $a$, $t$, and $c$ form configuration ($ii$) as $a \leq t < b < c$.
    Therefore, assume $t < a$.
    For $b$ to be left inserted into $S$, the element $g$ that was right inserted into $S$ must be less than $b$ by Lemma~\ref{lem:insertion_size}.
    Furthermore, before right insertion of $g$ we must have $b\notin S^{\uparrow}$, else $a$, $b$, and $c$ would already form configuration ($ii$).
    This implies the insertion rule for $g$ is (R3) and $b$ is in the row below $S$.
    Then (R3) guarantees $b$ is below the entry $h$ to the right of $g$ and $h > b$.
    If $h < c$, then $a$, $h$ and $c$ form configuration (ii).
    Therefore $h \geq c$, in which case $b$, $c$ and $h$ from configuration (iv/v).
    \\

    \noindent\textbf{Case (ii, c, R)}: Impossible as $c > b$ meaning $c\in S^{\downarrow}$.\\

    \noindent\textbf{Case (ii, c, L)}: 
    If $t > b$, then $a\leq b < t$ forms configuration ($ii$).
    Therefore, assume that $t\leq b$.
    For $c$ to be left inserted into $S$, it must be the output of the right insertion of some $h$ into $S$ with $h<c$ by Lemma~\ref{lem:insertion_size}.
    Further, we must have $b < h$ as otherwise $h$ would bump an less than $c$ as $t \leq b$ implies $b\in S^{\uparrow}$.
    Let $g$ be the element that was left inserted into $R$ to output $h$.
    By Lemma~\ref{l:bumpingpath} we have that whatever $g$ bumps is weakly right of the position where $a$ will be (if not there already).
    Therefore, we have that $g \leq a$ as it either bumps $a$, an element right of $a$ in $R^{\downarrow}$, or is $a$.
    By Lemma~\ref{lem:insertion_size} we have that $g > h$, however we also have that $g \leq a \leq b \leq h$.
    Thus, we could not have $c$ bump $t\leq b$.\\

    \noindent\textbf{Case (iii, v, R)}: As $t$ bumps $v$ not $z$ we must have that $t < z$.
    Therefore, $x < y \leq z$ and $t < z$ forms configuration ($iii$).\\

    \noindent\textbf{Case (iii, v, L)}: As $t < v$, we have that $x < y \leq z$ and $t < z$ forms configuration ($iii$).\\

    \noindent\textbf{Case (iii, x, R)}: As $z\in R^{\uparrow}$ and either $t < y$ or $t\in S^{\downarrow}$ we have that $y$ witnesses $z$ over $t$.\\

    \noindent\textbf{Case (iii, x, L)}: As $z\in R^{\uparrow}$ and $t\in S^{\downarrow}$ we have that $y$ witnesses $z$ over $t$.\\

    \noindent\textbf{Case (iii, y, R)}: If $t\leq z$, then $x$, $t$, $z$, and $v$ form configuration ($iii$).
    Therefore, assume $t > z$.
    In order for $t$ to be bumped by $y$ we require that the element $s$ directly left of $t$ is less than $y$.
    Further, $s \leq x$, as otherwise $s$ forms configuration ($iii$) with $x$, $z$, and $v$.
    If $s$ is not $x$ (in position), then $x\in S^{\downarrow}$ so $s$ witnesses $z$ over $x$.
    Therefore, $x$ is $s$, i.e., $x$ must be directly left of $t$.
    In order for $y$ to be right inserted into $S$ we require $y\in R$ or $y\in S$ before the prior insertion step.
    If $y\in R$ then $v < z$ implies $z \in R\uparrow$ so $y$ would be left of $z$.
    Therefore $x$, $y$, and $z$ would form configuration ($iv/v$), so we can assume $y\in S$.
    However, by definition of (L3) this would require that $y\in S^{\uparrow}$ which cannot happen as $x < y < t$ with $x$ and $t$ next to each other.\\

    \noindent\textbf{Case (iii, y, L)}: This requires that $t\in S^{\downarrow}$, but that would mean $x\in S^{\downarrow}$.
    This is a contradiction as $y$ could not bump $t$ with $x > t$ left of $t$.\\

    \noindent\textbf{Case (iii, z, R)}: As $t > z$ we have that $x$, $y$, $t$, and $v$ form configuration ($iii$).\\

    \noindent\textbf{Case (iii, z, L)}: As $v < z$ this would require left inserting $z$ into $R^{\uparrow}$ which is impossible.\\

    \noindent\textbf{Case (iv/v, x, R)}: If $t < y$, then $t < y \leq z$ forms configuration ($iv/v$).
    Therefore, assume that $t\geq y$.
    Let $g$ be the element left inserted in the step before and let $h$ be the element right inserted in the step before that.
    By Lemma~\ref{l:bumpingpath} $h$ inserts into a location weakly right of the column in $R$ containing $z$.
    Thus, $h\geq z$ as either $h=z$ or $h$ bumps an element right of $z$.
    Then by Lemma~\ref{lem:insertion_size} we have that $h<g$, thus $g > z$.
    
    If $y\in R^{\downarrow}$, since $g > z \geq y$ we see $g$ must bump an element $\geq y$ unless we are in case (L3), hence $g$ is already in $R\downarrow$ in column $j$.
    If we are in case (I), necessarily $x = \max(B)$ as elements of $A$ are $\geq y$, so $y$ is the dip of $R$ and the element left of $y$ is already $g$.
    Since $x = \max(B)$, either $x$ is immediately below $y$ or in $S\uparrow$.
    In the former case, $y$ is positive since it is the dip so $t$ witnesses $y$ over $x$.
    In the latter case, $x$ would not bump $t$.
    If we are in case (II), note that $r_{j+1} \geq y > x$ so we must output $x = s_{j+1}$.
    However, then $x \in S \uparrow$, which would again prevent $x$ from bumping $t$ when we right insert $x$.
    
    Therefore, we must have $y\in R^{\uparrow}$ with some $v <y$ directly left of $y$.
    Let $u$ be the element directly below $y$.
    We have that either $u\in S^{\downarrow}$ or $u\in S^{\uparrow}$ with $u < x$ so that $x$ does not bump $u$.
    If $u\in S^{\downarrow}$, then $x$ could not be in $R$ or in $S$ right of $u$ before left insertion of $g$ as otherwise $x$ would witness $y$ over $u$.
    Therefore, in this case we also have $u < x$.
    Thus, after right insertion of $x$ we have that $u < x < y$ and $v < y$ forms configuration ($iii$).
    As shown in \textbf{Case (iii, y, R)} this cannot happen.\\

    \noindent\textbf{Case (iv/v, x, L)}: As $t < x$ we have that $t$, $y$, and $z$ form configuration ($iv/v$).\\

    \noindent\textbf{Case (iv/v, y, R)}: As $t\in R^{\uparrow}$ and left of $z$ we have that $t < z$.
    Then as $t > y > x$ we have that $x$, $t$, and $z$ form configuration ($iv/v$).\\

    \noindent\textbf{Case (iv/v, y, L)}: If $t > x$, then as $t < y \leq z$ we have that $x$, $t$, and $z$ form configuration ($iv/v$).
    Therefore, assume that $t \leq x$.
    Before left insertion of $y$ we have $y\not\in R$ as otherwise $x$, $y$, and $z$ would already form configuration ($iv/v$).
    Thus, $y\in S$.
    Let $g$ be the element that was right inserted in order for us to end up left inserting $y$.
    Let $h$ be the element directly below where $g$ ends up after being inserted.
    By definition of (R3) we have $y > h$ and $x \leq h$ in order for $y$ to be the element left inserted instead of $x$.
    If $h\in S^{\uparrow}$ we must have $h$ weakly right of $x$ and thus $g$ weakly right of $z$, so $y$ is strictly right of $x$ and witnesses $z$ over $x$.
    If $h\in S^{\downarrow}$, then $g$ over $h$ would be witnessed by $x$ unless $h = x$.
    In that case, we also have $z = g\in R^{\uparrow}$ and $t < x$ as if $t$ is left of $h\in S^{\downarrow}$.
    Therefore $t$, $h$, $z$, and the element directly left of $z$ form configuration ($iv/v$).\\

    \noindent\textbf{Case (iv/v, z, R)}: As $t > z$, we have that $x$, $y$, and $t$ form configuration ($iv/v$).\\

    \noindent\textbf{Case (iv/v, z, L)}: As $t\in R^{\downarrow}$ and is right of $y$ we have that $y > t$.
    However, then we could never bump $t$ by left inserting $z$.

\end{proof}

Two words with the same Kra\'skiewicz tableaux $P$ are necessarily reduced words for the same signed permutation.
Similarly, the Kra\'skiewicz--Hecke insertion tableau for a word determines it 0-Hecke product.

\begin{proposition}
    \label{p:reading-word}
    Let $T$ be a SDT and $\rho(T)$ be its reading word.
    For $a\in \mathbb{N}$ and $T' = T \leftarrow a$ we have that $\rho(T)a \equiv_{H} \rho(T')$.
\end{proposition}

\begin{proof}
    The result follows if when $a$ is row inserted into the two-row SDT $RS$ bumping $b$, we have $SRa \equiv_{H} SbR'$.
    Throughout this proof, we fix notation: $p$ is the dip of $S$; $q$ is the dip of $R$; $i$ and $j$ are the indices defined in right and left row insertion, respectively.
    
    First, we treat the case where $a = 0$ and $r_i = 0, r_{i+1} = 1$ in $R$.
    In this case, we will show $b = 0$. 
    \begin{enumerate}
        \item [(R2)/(L1)] Assume $r_{i-1} > 1$ so that $r_{i+1}$ is the only 1 in $R$.
        When $0$ is inserted $0$ will bump $r_{i+1} = 1$ which in turn will bump $r_{i} = 0$.
        The overall result is $R'$ will be $R$ with $r_{i}$ and $r_{i+1}$ swapped and the output to the next row is $0$.
        To see that $SR0 \equiv_{H} S0R'$ commute $0$ up to $r_{i+1} = 1$, then commute $r_{i}=0$ up to $S$ as $r_{k} > 1$ for $k < i$.
        This effectively swaps the values of $r_{i}$ and $r_{i+1}$, thus showing that $SR0 \equiv_{H} S0R'$.

%        When $1$ is left inserted, we get $r_{j}=1$ and $r_{j+1}=0$ with $i=j$ (as the values $0$ and $1$ were swapped).
%        Suppose that $j < p$ and $s_j > 1$.
%        However, then in the original tableau as $s_j \in S^{\downarrow}$ we have $s_j < r_{i+1}$ which by the definition of a standard decomposition tableau means that there are no entries in $[r_{i+1}, s_j)$ between $s_j$ and $r_{i+1}$.
%        Therefore the sets $A$ and $B$ in the definition of $(L3)$ are both empty.
%        Then note that for condition $(c)$ in the definition of $(L3)$ to be true we require that $s_{j+1}=0$, which then means that $j+1 = p$ resulting in condition $(b)$ being false.
%        Therefore, we get that $a' = r_{j+1} = 0$.
     
        \item [(R3)/(L2)] Assume $r_{i-1} = 1$. Note that in this case (R3) will always output $r_{i+1} = 1$ and (L2) will always output $0$, both leaving $R$ unchanged.
        To see $SR0 \equiv_{H} S0R$, first commute $a = 0$ up to $r_{i+1} = 1$, then perform a long braid with $r_{i-1} = 1$, $r_{i} = 0$, $r_{i+1} = 1$, and $a = 0$ to get $0$ on the other side of $101$.
        Next, commute the left $0$ past the rest of $R^\downarrow$.

%        Then right insertion will trigger $(R3)$ and as $r_{i+1}=1$ we have that $(s_{i}, r_{i+1})$ must be the empty set meaning that we will left insert $r_{i+1}=1$.
%        For the same reasons outlined in the previous case left insertion will result in inserting $0$ to the next row.
      
    \end{enumerate}

    Now assume either that $a \neq 0$ or $R$ does not contain $0$ followed by $1$.
    Let $a'$ be the entry bumped when right inserting $a$, and recall that $a' > a$.
    First we will show $SRa \equiv_{H} SR^{\downarrow}a'R''^{\uparrow}$.
    \begin{enumerate}
        \item [(R1)] Let right insertion of $a$ into $R$ trigger $(R1)$.
        Then $r_i \geq a+1$ is not the dip.
        Therefore, we can commute $a$ up to $r_i$ since all entries to its right are greater than $r_i$.
        Since $r_{i-1} < a$ as $i$ is not the dip, we can commute $r_i$ past $r_q$, giving us that $SRa \equiv_{H} SR^{\downarrow}a'R'^{\uparrow}$.
        \item [(R2)] Let right insertion of $a$ into $R$ trigger $(R2)$, so $r_i = a$.
        Recall $x,y \in \mathbb{N}$ with $\{x,y\} \neq \{0,1\}$ we have the short braid relation $xyx\equiv_{H} yxy$.
        Therefore we can commute $a$ up to $r_{i+1}$, to get the desired result, taking $r_i \in R^{\downarrow}$.
        \item [(R3)] Let right insertion of $a$ into $R$ trigger $(R3)$. We split into cases, depending on where $a'$ is located.
        When $a'=r_{i+1}$, the argument is the same as (R2) except after braiding we commute the left instance of $a' = r_{i+1}$ past the dip.

        Otherwise, $a' \in (s_i, r_{i+1})$, which implies $s_i < r_{i+1} - 1$ so we cannot have $r_{i+1} = a+1$ without $a'$ witnessing $s_{i} < r_{i} = a$.
        Therefore, we can commute $a$ up to $r_{i}=a$ and annihilate $a$.
        It remains to get a copy of $a'$ directly after $r_q$.

        Suppose that $a' \in S$, then $a'$ is to the right of $s_i$, which implies that the element above it is greater than or equal to $r_{i+1}$.
        Therefore there are no elements in $(a'',r_{i+1}]$ between $a'$ and $r_i$.
        If there is no $a'-1$ between $a'$ and $r_q$, then a copy of $a'$ could commute all the way up to $r_q$ and we are done.
        Suppose there is an $a'-1$ between $a'$ and $r_q$ (note this means that $a'-1\neq 0$), this requires that $a'-1\not\in (s_i, r_{i+1})$ by definition of $a'$ meaning that $a' = s_i+1$.
        Then as there are no elements in $(a',r_{i+1}]$ between $a'$ and $r_i$ we can commute the pair of $s_i$ and $a'$ up to the $a'-1$, possibly annihilating an $a'$ immediately preceding the $a'-1$ in $R$.
        Then duplicate $s_i$ and perform a short braid relation with the $s_i$, $a'$, and $a'-1$.
        Finally, commute back $a'-1$ and $a'$, possibly leaving behind a duplicate of the left $a'$, and commute forward up to the dip the extra copy of $a'$.
            
        Suppose that $a'\in R^\downarrow$, then $a'$ is to the left of $r_i$ in $R$.
        If $a'-1 \not\in R^{\downarrow}$, we can copy $a'$ and commute the copy up to $r_q$.
        If $a'-1\in R^{\downarrow}$, then by definition of $a'$ we must have that $a'-1 = s_i$, as otherwise we would have a different value for $a'$.
        Further, there can be no $s_i-1$ in $S$ after $s_i$, as the element above such a $s_i - 1$ would be witnessed by $a'$.
        The case where $s_{i}+1$ is to the right of $s_i$ in $S$ is treated above.
        Therefore, we can duplicate $s_i$ in $S$, then commute the duplicate up to $a'$, short braid the duplicate, $a'$, and $a'-1$, and then commute the copy of $a'$ made up to $r_q$.
    \end{enumerate}
    Now we left insert whatever $a'$ was the output from right insertion.
    In right insertion we addressed the case where $a = 0$ and $R$ contains $0$ followed by $1$, so for left insertion we already have the case where $a' =  1$ and $R''$ contains $1$ followed by $0$.
    We now show $SR^{\downarrow}a'R''^{\uparrow}\equiv_{H} SbR'$.
    \begin{enumerate}
        \item [(L1)] Let left insertion of $a'$ into $R''$ trigger $(L1)$.
        This then means that $r_j$ is the greatest element in $R''^{\downarrow}$ less than $a'$ and $a'\not\in R''^{\downarrow}$.
        Therefore, we can commute $a'$ up to $r_j$ and then commute $r_j$ past the rest of $R''^{\downarrow}$ for the desired outcome.
        \item [(L2)] Let left insertion of $a'$ into $R''$ trigger $(L2)$, so $a' = r_j$.
        We have already treated the case $a' = 1$.
        First assume $r_{j+1} = a'-1$ and $j+1$ is not the dip of $R''$.
        In this case condition $(a)$ in $(L2)$ is always false, so either $b = \max A$ or $b = r_{j+1}$.
        When $b = \max A$, for $A$ to be nonempty we have $r_j < s_j$ as $r_{j+1} = r_{j}-1$.
        However as $r_j \in R''^{\downarrow}$ and $s_j \in S^{\downarrow}$, so $b$ would witness $r_j$ over $s_j$.
        Therefore, we must have $b= r_{j+1} = a'-1$.
        Note $b > 0$ since $b= r_{j+1}$ is not the dip.
        Therefore we can braid $r_{j}$, $r_{j+1}$, and $a'$ to get a copy of $b = a'-1$ on the other side, which can then commute past the rest of $R''$.\\
        Now we consider the cases where $r_{j+1} \neq a'-1$ or is the dip.
        In both cases $a'$ can freely commute to $r_{j} = a'$ to be annihilated as the insertion element is considered to be past the dip.
        Thus, we aim to show that $SR''\equiv_{H} Sa'R'$.\\
        First, suppose that $b = \max A$ or $b = \max B$.
        If there is no $b + 1$ between $b$ and the end of $S$ we can freely commute $b$ to the end of $S$ and be done.
        Otherwise, since $b+1 \notin A$ or $B$ (depending on the case), we see $b + 1 \geq s_j$, but $b < s_j$  so we have $b + 1 = s_j$ .
        If $b\in R$, necessarily $b > a'$ meaning that $s_j > a' = r_j$, which $a'$ would witness.
        Therefore, $b\in S$.
        We create a copy of $s_j$ then commute the copy to $b = s_j-1$.
        Then perform a braid relation and commute out of $S$ the extra copy of $a'$.\\
        Now suppose that $b = s_{j+1}$, $b \neq \max A$, and $a \neq \max B$.
        This means $s_{j+1} \in S^{\uparrow}$ by condition $(b)$ from $(L2)$.
        Note condition $(c)$ from $(L2)$ prevents us from having $s_{j+1}+1 \in S^{\uparrow}$ as either $s_{j+1}+1$ witnesses $s_{j+1} < r_{j+1}$ or $s_{j+1} < r_{j} < s_{j+2}$, meaning $s_{j+2} > s_{j+1}+1$.
        Therefore, a copy of $b$ made by $s_{j+1}$ can freely commute out of $S$.\\
        Finally, when $b = r_{j+1}$ since we have assumed either $r_{j+1}$ is the dip or $r_{j} \neq b + 1$ there is no obstruction to commuting a copy of $b$ past the rest of $R''$.
    \end{enumerate}
    This completes our proof.
\end{proof}

For $\lambda = (\lambda_1,\dots,\lambda_k)$ a strict shape, we will see for words in $\HH(w_\lambda)$ that $Q_KH$ and $\res$ are the same map.
Towards that end, let $(c_1 < \dots < c_{\lambda_1 -k})$ so that $[0,\lambda_1-1] = \{\lambda_1,\dots,\lambda_k,c_1,\dots,c_{\lambda_1 - k}\}$ and define $P^\lambda$ as the tableau with $i$th row
\[
\lambda_{i+1},\dots,\lambda_k, 0(=\lambda_{k+1}),c_1,\dots,c_{\lambda_i - k +i},
\]
so that $P_\lambda$ has shape $\lambda$.
By construction, note that every column of $P^\lambda$ has the same entry.
Moreover, we obtain $P^{(\lambda_2,\dots,\lambda_k)}$ from $P^\lambda$ be removing the first row.
For example, with $\lambda = (5,3,2)$, we have
\[
P^{(5,3,2)} = \begin{ytableau}
    3 & 2 & 0 & 1 & 4\\
    \none & 2 & 0 & 1\\
    \none & \none & 0 & 1
\end{ytableau}\,,
\]
while the second and third rows give $P^{(3,2)}$ and the last is $P^{(2)}$.

Note that $\lambda_i - 1$ is the largest entry in the $i$th row of $P^\lambda$.
When $(i,j)$ is an outer corner of $D_\lambda$, we then have $P^\lambda_{ij} = \lambda_i - 1$ since $\lambda_{i+1} < \lambda_i - 1$.

\begin{theorem}\label{thm: kr_undo_res}
Let $T\in \ShSVT(\lambda)$, and let $\bfa = \res(T)$.
Then $KH$ maps $\bfa$ to $(P^\lambda,T)$.
\end{theorem}

\begin{proof}
Since $\bfa = \res(T)$, we see $\bfa \in \HH(w_\lambda)$.
For $\bfa = (a_1,\dots,a_p)$ and $j \in [p]$, note that $\bfa[j] := (a_1,\dots,a_j) \in \HH(w_\mu^{(j)})$ for some strict shape $\mu^{(j)}$.
Since $w_\mu$ is vexillary for any $\mu$, Theorem~\ref{t:vexillary} shows $P_{KH}(\bfa[j])$ is unique.
We can now prove our result by downward induction on $k$.
The base case is immediate since any word $\bfa$ for $w_\lambda$ has $a_1 = 0$, so the $j = 1$ case follows.

Assume the result holds for up to $k$, so $(P^{\mu^{(k)}},T\mid_{[k]})$ is the image of $\bfa[k]$ under $KH$.
The next step is to insert $a_{k+1}$ into $P^{\mu^{(k)}}$ to obtain $P'$, with resulting recording tableau $Q'$.
There are two cases.

\noindent \textbf{Case 1:} $P' = P^{\mu^{(k)}}$.
In this case, we see $w_{\mu^{(k)}} \circ a_{k+1} = w_{\mu^{(k)}}$, so $k+1 \in T_{ij}$ and is not minimal in this cell since $(i,j) \in D_{\mu^{(k)}}$.
Then $(i,j)$ is an outer corner of $\mu^{(k)}$ so $j = \mu^{(k)}_i + i -1$.
Then $a_{k+1} = (j-i) = \mu^{(k)}_i - 1$.

We $a_{k+1}$ inserts into $P_{\mu^{(k)}}$ without changing it, terminating in row $j$.
To see this, for $R_\ell$ the $\ell$th row of $P^\lambda$ note $a_{k+1} \in R_\ell^\uparrow$ for $\ell \leq i$.
For the first step of insertion, we are in (R3).
If $i = 1$, we see $a_{k+1}$ is the largest entry in the first row, so we terminate in row $1$ as desired.
Otherwise, $a'$ is the smallest entry in $R_1^\downarrow$ greater than $a_{k+1}$, which is $\mu^{(k)}_i$.
Left inserting $\mu^{(k)}_i$ uses (L2), which bumps $a'' = \mu^{(k)}_i - 1 = a_{k+1}$ as this value is maximal in the appropriate interval.
We are now inserting $a_{k+1}$ into the second row of $P^{\mu^{(k)}}$, which is $P^\nu$ for $\nu$ obtained from $\mu^{(k)}$ by deleting the first row.
The result then follows by the inductive hypothesis.

\noindent\textbf{Case 2:} $P' \neq P^{\mu^{(k)}}$.
By the previous case, we see $\mu^{(k+1}) \neq \mu^{(k)}$.
Let $(i,j)$ be the unique cell in $D_{\mu^{(k+1)}} \setminus D_{\mu^{(k)}}$, so $j =  \mu^{(k+1)}_i +i -1 = \mu^{(k)}_i + i$.
Then $a_{k+1} = j-i = \mu^{(k)}_i$.

If $i =1$, when inserting $a_{k+1}$ we apply (R1) to append it to the end of the first row, giving $P' = P^{\mu^{(k+1)}}$ as desired.
Otherwise, let $R$ be the first row of $P^{\mu^{(k)}}$ and $R'$ the result after $KH$--inserting $a_{k+1}$ into $R$.
Since $\mu^{(k)}_i$ is in $R^\downarrow$ and $\mu^{(k)}_i + 1$ is in $R^{\uparrow}$, by applying rules (R1) and (L1) we swap these to obtain $R'$, inserting $\mu^{(k)}_i$ into the next row.
Applying induction on $i$ as in the previous case, this process will terminate by inserting $a_{k+1}$ at the end of the $i$th row while swapping $\mu^{(k)}_i$ with $\mu^{(k)}_i{+1} = \mu^{(k+1)}_i$ as desired.
This completes our proof.
\end{proof}

Case 2 in the above argument appears in the second author's thesis~\cite[Thm. 4.4~(a)]{hamaker2014bijective}.
It implies the equivalent result for tableaux that are not set--valued, where Kra\'skiewicz insertion is applicable.
It has not been previously published.

\subsection{Inverse insertion}

In order to show that Kra\'{s}kiewicz--Hecke insertion is a bijection we will first show that the map is injective and then create an explicit inverse map.

\begin{proposition}
    Let $M$, $N$, $T$ be SDTs and $x, y, z\in \mathbb{N}$.
    If right inserting $x$ into $M$ at row $k$ and right inserting $y$ into $N$ at row $k$ both output left inserting $z$ into $T$ at row $k$, then $M=N$ and $x=y$.
    Similarly, if left inserting $x$ into $M$ at row $k$ and left inserting $y$ into $N$ at row $k$ both output right inserting $z$ into $T$ at row $k+1$, then $M=N$ and $x=y$.
\end{proposition}

\begin{proof}
    Let $M$, $N$, $T$ be SDTs and $x, y, z\in \mathbb{N}$.
    Suppose right inserting $x$ into $M$ at row $k$ and right inserting $y$ into $N$ at row $k$ both output left inserting $z$ into $T$ at row $k$.\\
    \textbf{Case 1}: Suppose inserting $x$ into $M$ triggers rule (R1).
    Let $m_i = z$ be the element $x$ bumped.\\
    \textbf{Case 1.1}: Suppose inserting $y$ into $N$ would trigger rule (R1).
    Then in order for the resulting tableau to be $T$ and output to be $z$ we require that position $i$ is where bumping occurs, $n_i = z$, and $y = x$.
    Thus, $M = N$ and $y = x$.\\
    \textbf{Case 1.2}: Suppose inserting $y$ into $N$ would trigger rule (R2).
    Then $T$ would have two $y$s in row $k$, however as inserting $x$ into $N$ triggered (R1) we know row $k$ in $T$ is strictly unimodal.
    Therefore, inserting $y$ into $N$ cannot trigger (R2).\\
    \textbf{Case 1.3}: Suppose inserting $y$ into $N$ would trigger rule (R3), note this means that $N = T$.
    By definition of (R1) we have that $m_{i+1} > z$.
    Therefore, we require that $y = x$ as if $y < x$ then we could not output $z > x$ and if $y > x$ we would also require $y > z$.
    Further as $m_{i+1} > z$ we cannot output $m_{i+1}$, meaning that the output $z$ is somewhere right of $s_i$ (the element below $m_i$) in row $k+1$ or left of $m_i$ in row $k$ of $N$.
    However, we must have that this $z$ is also in $M$, but then $z$ would witness $s_i < z = m_i$ in $M$ meaning $M$ was not a SDT.
    Therefore, inserting $y$ into $N$ cannot trigger (R3).\\
    \textbf{Case 2}: Suppose inserting $x$ into $M$ triggers rule (R2).
    Let $m_i = x$ and $m_{i+1} = z$ where $i$ is the location bumping occurred at.\\
    \textbf{Case 2.2}: Suppose inserting $y$ into $N$ would trigger rule (R2).
    In order to end up with output tableau $T$ we require that $x=y$ and in order to output $z$ we require that the element bumped by $x$ is $z$.
    Therefore, $N=M$ and $y=x$.\\
    \textbf{Case 2.3}: Suppose inserting $y$ into $N$ would trigger rule (R3), note this means that $N = T$.
    However, $T$ is a pseudo-SDT and $N$ is not.
    Therefore, this cannot happen.
\begin{comment}
\[
    \begin{ytableau}\none & 3 & 2 & 0 & 3 & \none[\;\leftarrow 0]\\ \none &\none & 0 & 1 & \none\\ \none &\none &\none & \none\\ \end{ytableau}\qquad
    \begin{ytableau}\none[2\rightarrow\;] & 3 & 2 & 0 & 3 & \none\\ \none &\none & 0 & 1 & \none\\ \none &\none &\none & \none\\ \end{ytableau}
\]
\[
    \begin{ytableau}\none & 3 & 0 & 2 & 3 & \none[\;\leftarrow 0]\\ \none &\none & 0 & 1 & \none\\ \none &\none &\none & \none\\ \end{ytableau}\qquad
    \begin{ytableau}\none[2\rightarrow\;] & 3 & 2 & 0 & 3 & \none\\ \none &\none & 0 & 1 & \none\\ \none &\none &\none & \none\\ \end{ytableau}
\]
\end{comment}

    \noindent \textbf{Case 3}: Suppose inserting $x$ into $M$ triggers rule (R3).\\
    \textbf{Case 3.3}: Suppose inserting $y$ into $N$ would trigger rule (R3), note this means that $N = T = M$.
    Suppose $x < y$.
    By Lemma~\ref{lem:insertion_size} we have that $z > y$ from inserting into $N$, however by definition of (R3) for inserting into $M$ we have that $z\leq y$ as the output is at most the value directly right of the copy of the input.
    Therefore, we cannot have $x \neq y$.\\

    Let $M$, $N$, $T$ be SDTs and $x, y, z\in \mathbb{N}$.
    Suppose left inserting $x$ into $M$ at row $k$ and left inserting $y$ into $N$ at row $k$ both output left inserting $z$ into $T$ at row $k$.\\
    \textbf{Case 1}: Suppose inserting $x$ into $M$ triggers rule (L1).
    Let $m_j = z$ be the element $x$ bumped.\\
    \textbf{Case 1.1}: Suppose inserting $y$ into $N$ would trigger rule (L1).
    In order for the resulting tableau to be $T$ and output $z$ we require that position $j$ is where bumping occurs, $n_j = z$, and $y=x$.
    Thus, $M = N$ and $y = x$.\\
    \textbf{Case 1.2}: Suppose inserting $y$ into $N$ would trigger rule (L2), note this means that $N=T$.
    By definition of (L1) we have that $x > z$.
    Suppose that $y < x$, then in order for $N=T$ we require that in $M$ y is right on $m_j = z$ meaning $z > y$.
    However this contradicts Lemma~\ref{lem:insertion_size}.
    Suppose $y > x$.
    Note that as $N=T$ we have that $x$ is in $N$ and by Corollary~\ref{cor:no_dip} $x$ cannot be the dip.
    Thus, by definition of (L2) the output must be greater or equal to $x$ meaning we cannot be in this case.
    Suppose $y = x$.
    If $j < p$, the dip of the row $k+1$ of $N$, then we must have $z < s_j$.
    However, then the $z$ we output from inserting $x$ into $N$ serve as a witness in $M$ meaning $M$ is not a SDT.
    Therefore, we have that $j \geq p$.
    This then requires that $s_{j+1} = z$ and $r_{j+2}, s_{j+2} > x$ (as we know $r_{j+1} > z$ due to $z$ being the element bumped when $x$ is inserted into $M$).
    However, this then prevents another $x$ from existing in any location that would allow it to end being left inserted into $N$ meaning this cannot happen.
    Therefore, inserting $y$ into $N$ cannot trigger rule (L2).\\
    \textbf{Case 2}: Suppose inserting $x$ into $M$ triggers rule (L2).\\
    \textbf{Case 2.2}: Suppose inserting $y$ into $N$ would trigger rule (L2), note this means that $M=N=T$ so suppose $x\neq y$.
    Without loss of generality suppose $x < y$.
    By Corollary~\ref{cor:no_dip} $x$ cannot be the dip, therefore when inserting $y$ we cannot output an element less than $x$.
    However, $z < x$ meaning this case cannot happen.
\end{proof}

To prove that Kra\'skiewicz--Hecke insertion is a bijection, we demonstrate an explicit inverse map.

\begin{definition}
    \label{d:inverse-row-insertion}
    We define \emph{inverse Kra\'skiewicz--Hecke row insertion} in two steps, each using an input $a \in \mathbb{N} \cup \{\infty\}$ and a strict decomposition tableau with distinguished row $R$ above row $S$.
    Write $R = r_1 \dots r_\ell$ and $S = s_2 \dots s_k$ with $k \leq \ell$, $r_q$ the dip of $R$ and $s_p$ the dip of $S$.
    Note $k$ could be 1, in which case $S$ is empty.
    Set $r_0, s_1 = -\infty$, $r_{\ell+1} = \infty$ and $s_m = \infty$ for $m > k$.
    Note $r_i$ appears immediately above $s_i$ in $RS$.
    For any $d\in \N$, define $A_d = \{s_{d+1},\dots,s_k,r_1,\dots,r_{d-1}\}$.
    There are two forms of inverse insertion, \emph{left inverse insertion} and \emph{right inverse insertion}.
    First we define left inverse insertion, for which we require the terminating tableau from right inserting $a$ into row $S$ to be a SDT.

    Let $1\leq j < q$ be maximal such that $r_j > a$.

    \begin{enumerate}
        \item[(I)] If $j < p$:
        \begin{enumerate}
            \item If $j+1 < q$ and $a = \max \{r_{j+1}\}\cup ( A_j \cap (r_{j+1}, s_j))$,  then go to (L2)
            \item If $j+1 \geq q$ and $a = \max \{r_{j+1}\}\cup ( A_j \cap (-\infty, s_j))$,  then go to (L2)
        \end{enumerate}
        \noindent Otherwise, go to (L1)
        
        \item[(II)] If $j \geq p$:
        \begin{enumerate}
            \item If the following are both true, go to (L2)
        \begin{enumerate}
            \item $j+1 \geq q$,
            \item $a = s_{j+1} < r_{j+1}$ or $r_{j+2}, s_{j+2} > r_j > s_{j+1} = a$.
        \end{enumerate}
        \item Otherwise, if $a = r_{j+1}$, and either of the following are true, go to (L2)
        \begin{enumerate}
            \item $r_{j+2} \leq r_j$, or
            \item $s_{j+1} < r_{j}$ and there exists $k\geq j+2$ such that $s_k\in (s_{j+1}, r_{j}]$,
        \end{enumerate}
        \end{enumerate}
    \noindent Otherwise go to (L1).
    \item[(L2)] \label{rule.l2} Set $a' = r_{j}$ and $R'' = R$.
    \item [(L1)] Set $a' = r_{j}$ and replace $r_{j}$ with $a$ to get $R''$.
    \end{enumerate}

    We now define inverse right insertion, for which we require the terminating tableau from left inserting $a'$ into $R''$ to be a SDT. Let $q \leq i \leq \ell$ be maximal such that $r_{i} \leq a'$:
    \begin{enumerate}
        \item[(R3)] (I) If $s_i < a'$ and
        \begin{equation*}
            a' = \min ( (s_i, r_{i+1}) \cap A_i ),
        \end{equation*}
        then set $R' = R''$ and $b = r_i$.\\
        (II) If $a' = r_{i}$, $r_{i} \geq r_{i-2}$, and
        \begin{equation*}
            a' = \min ( \{r_{i}\} \cup [ (s_{i-1}, r_{i}) \cap A_{i-1} ]),
        \end{equation*}
        then set $R' = R''$ and $b= r_{i-1}$.
        \item[(R1/R2)] Else, set $b=r_i$ and create $R'$ by setting $r_{i}=a'$.
        %\item[(R1/R2)] Else if $r_i < a' < r_{i+1}$, then set $b = r_i$ and create $R'$ by setting $r_{i}=a'$.
        %\item[(R2)] Else if $r_i = r_{i-1}$ then set $b = r_{i}$ and create $R'$ by setting $r_{i} = a'$.
    \end{enumerate}

\end{definition}

\begin{example}
    We present examples of the various steps of the inverse insertion.
    \begin{enumerate}
        \item Left inverse inserting $1$ into the top row of the tableau below is undefined as if you right insert $1$ into the row below the tableau is no longer an SDT.
        \[
        \begin{ytableau}\none[1\leftarrow \;\;\;]&4&3&1&0&1&3&\none\\\none&\none&3&2\\\end{ytableau}
    \begin{ytableau}\none&4&3&1&0&1&3\\\none&\none&3&2&\none[\;\;\;\leftarrow 1]\\\end{ytableau}
    \begin{ytableau}\none&4&3&*(red)1&0&1&3\\\none&\none&3&*(red)2&*(red)1\\\end{ytableau}
        \]\vspace{.25cm}

    \item Here $r_j = 2$. None of the conditions for (L2) are met, so we set $a'=r_j=2$ and swap $a$ and $r_j$ to get $R'$.
    \[
    \begin{ytableau}\none[0 \leftarrow\;\;\;]&3&2&1&\none\\\end{ytableau}
    \begin{ytableau}\none&3&0&1&\none[\;\;\;\rightarrow 2]\\\end{ytableau}
    \]\vspace{,25cm}

    \item Here $r_j = 3$. As $j = 2 < p = 3$, $j+1 = 3 < q = 4$, and $\max \{1\} \cup [ (1,3)\cap\{2,4\}] = 2 = a$, we set $a' = r_j$.
    \[
    \begin{ytableau}\none[2\leftarrow \;\;\;]&*(green)4&3&1&0&1&3&\none\\\none&\none&3&*(green)2\\\end{ytableau}
    \begin{ytableau}\none&4&3&1&0&1&3&\none[\;\;\;\rightarrow 3]\\\none&\none&3&2\\\end{ytableau}
    \]\vspace{.25cm}

    \item Here $r_j = 1$. As $j = 2 < p = 3$ and $j+1 = 3 \geq q = 3$, and $\max \{0\} \cup [ (-\infty, 3) \cap \{0, 4\} ] = 0 = a$, we set $a' = r_j$.
    \[
    \begin{ytableau}\none[0\leftarrow\;\;\;]&*(green)4&1&0&1&\none\\\none&\none&3&*(green)0\\\end{ytableau}
    \begin{ytableau}\none&4&1&0&1&\none[\;\;\;\rightarrow 1]\\\none&\none&3&0&\none\\\end{ytableau}
    \]\vspace{.25em}

    \item Here $r_j = 4$. We are in case (II) (a) as $j + 1 = 3 \geq q = 3$, $j = 2 \geq p = 2$, and $r_{j+1} = 2 > s_{j+1} = 1 = a$. Thus, we set $a' = r_{j} = 4$.
    \[
    \begin{ytableau}\none[1 \leftarrow\;\;\;]&5&4&2&3&\none\\\none&\none&0&1&\none\\\end{ytableau}
    \begin{ytableau}\none&5&4&2&3&\none[\;\;\;\rightarrow 4]\\\none&\none&0&1\\\end{ytableau}
    \]\vspace{.25cm}

    \item Here $r_j = 2$. We are in case (II) (a) as $j + 1 = 3 \geq q = 3$, $j = 2 \geq p = 2$, and $r_{j+2}, s_{j+2} = 3,3 > r_j = 2 > s_{j+1} = 1 = a$. Thus, we set $a' = r_{j} = 2$.
    \[
    \begin{ytableau}\none[1 \leftarrow\;\;\;]&4&2&1&3&\none\\\none&\none&0&1&3\\\end{ytableau}
    \begin{ytableau}\none&4&2&1&3&\none[\;\;\;\rightarrow 2]\\\none&\none&0&1&3&\none\\\end{ytableau}
    \]\vspace{.25cm}

    \item Here $r_j = 2$. As $j = 2 \geq p = 2$ we are in case (II).
    As $a \neq s_{j+1}$ and $a = r_{j+1} = 1$, we consider case (b).
    As $0 = r_{j+2} \leq r_j = 2$ we set $a' = r_j$.
    \[
    \begin{ytableau}\none[1 \leftarrow \;\;\;]&4&2&1&0&\none\\\none&\none&2&3\\\end{ytableau}
    \begin{ytableau}\none&4&2&1&0&\none[\;\;\;\rightarrow 2]\\\none&\none&2&3&\none\\\end{ytableau}
    \]\vspace{.25cm}

    \item Here $r_j=3$. As $j = 1 \geq p = 2$ we are in case (II).
    As $a = 0 \not< r_{j+1} = 0$ and $r_{j+2} = 2 \not> 3 = r_j$ (ii) is false meaning we aren't in case (a).
    As $a = r_{j+1} = 0$, $s_{j+1} = 1 < 3 = r_{j}$, and $s_{3} \in (0, 3]$ we are in case (b) meaning we set $a' = 3 = r_{j}$.
    \[
    \begin{ytableau}\none[0 \leftarrow \;\;\;]&3&0&2&3&\none\\\none&\none&0&1\\\end{ytableau}
    \begin{ytableau}\none&3&0&2&3&\none[\;\;\;\rightarrow 0]\\\none&\none&0&1&\none\\\end{ytableau}
    \]\vspace{.25cm}

    \item Here $r_j=2$. As $j = 2 \geq p = 2$ we are in case (II).
    As $a = 0 \neq 1 = s_{j+1}$ (ii) is false meaning we aren't in case (a).
    As $r_{j+2} = 3 \not\geq 2 = r_{j}$ and there is not $k\geq j+2$ such that $s_k \in (1,3]$, we are also not in case (b).
    Therefore (L1) applies setting $a' = 2$ and replacing $r_j$ with $0$ to get $R'$.
    \[
    \begin{ytableau}\none[0 \leftarrow \;\;\;]&3&2&0&3&\none\\\none&\none&0&1\\\end{ytableau}
    \begin{ytableau}\none&3&0&0&3&\none[\;\;\;\rightarrow 2]\\\none&\none&0&1&\none\\\end{ytableau}
    \]\vspace{.25cm}

    \item Here $r_i = 0$. As $\{3,0\}\cap (0,1) = \emptyset$ and $a' \neq 3$ we are not in case (I).
    As $a' \neq r_i = 0$, we are also not in case (II).
    Therefore, we apply (R1/R2) by setting $b = 0$ and replacing $r_i$ with $2$.
    \[
    \begin{ytableau}\none\\\none&3&0&0&3&\none[\;\;\;\rightarrow 2]\\\none&\none&0&1\\\end{ytableau}\hspace{1em}
    \begin{ytableau}\none[0 \leftarrow ]\\\none&3&0&2&3&\none\\\none&\none&0&1\\\end{ytableau}
    \]\vspace{.25cm}

    \item Here $r_i = 1$. As $s_i = 1 < a = 2 < r_{i+1} = 3$ and $a = \min(\{3\}\cup [(1,3)\cap \{4,2,0\}]) = 2$, due to (I) we set $b=r_i =1$ and leave $R'$ unchanged.
    \[
    \begin{ytableau}\none\\\none&*(green)4&*(green)2&*(green)0&1&3&\none[\;\;\;\rightarrow 2]\\\none&\none&2&0&1\\\end{ytableau}\hspace{1em}
    \begin{ytableau}\none[1 \leftarrow ]\\\none&4&2&0&1&3&\none\\\none&\none&2&0&1\\\end{ytableau}
    \]\vspace{.25cm}

    \item Here $r_i = 3$. As $s_i = \infty$ we are not in case (I).
    As $a' = 3 = r_i \geq 0 = r_{i-2}$ and $a' = \min (\{3\} \cup [(1,3)\cap \{0,1,4\}])$, due to (II) we set $b = r_{i-1} = 3$ and leave $R'$ unchanged.
    \[
    \begin{ytableau}\none\\\none&*(green)4&*(green)1&*(green)0&*(green)1&3&\none[\;\;\;\rightarrow 3]\\\none&\none&1&0&1\\\end{ytableau}\hspace{1em}
    \begin{ytableau}\none[1 \leftarrow ]\\\none&4&1&0&1&3&\none\\\none&\none&1&0&1\\\end{ytableau}
    \]\vspace{.25cm}

    \end{enumerate}
    
\end{example}

We now show these inverse maps are well-defined.

\begin{theorem}
\label{t:invertible}
    Left and right inverse Kra\'skiewicz--Hecke row insertions are well-defined and are inverse to left and right Kra\'skiewicz--Hecke row insertions, respectively.
\end{theorem}

\begin{proof}
    Note that each inverse insertion step preserves unimodality of rows with the exception of (L1) when $a=r_{j+1}$.
    However, in that case the next inverse insertion step will restore unimodality of the row as it will always trigger (R1/R2).
    
    To see this we first aim to show that $i = j+1 = q$.
    Assume for the sake of obtaining a contradiction, that $r_{j+2} \leq r_{j}$. We discuss the following two cases.
    
    \noindent\textbf{(I): Suppose $j < p$}, where $p$ is the dip of $S$.
    Then as left inverse inserting $a$ triggers (L1) we must have that $a\neq \max\{r_{j+1}\}\cup A_j \cap (r_{j+1}, s_j))$ if $j+1 < q$ or $a\neq \max\{r_{j+1}\}\cup A_j \cap (-\infty, s_j))$ if $j+1 \geq q$.
    Therefore, there must exists some $x\in A_j$ such that $a < x < s_j$, take your $x$ to be the topmost, leftmost such element.
    
    Suppose $x\in R$, then as $x$ is left of $r_j$, $x$ would witness $r_j$ over $s_j$.
    
    Suppose $x\in S^{\downarrow}$, then if you would right insert $a$ into $S$, there would be an $a$ right of $x$.
    However, then that $a$ would witness $a$ over $x$.
    
    Suppose $x\in S^{\uparrow}$.
    Let $y$ be the element bumped by right inserting $a$ into $S$.
    Note that $y$ must be weakly left of $x$ as $a < x$, meaning $y \leq x < s_j$.
    Then as $s_{j+1} < a < y < s_j$ we have that left inserting $y$ into $S$ will bump $s_{j+1}$.
    However, then $a\in S$ will witness $a$ over $y$.
    Thus, we cannot have $j < p$.
    
    \noindent\textbf{(II): Suppose $j \geq p$}, then we left inverse inserting $a$ we find (II)(b)(i) is always true meaning we can never trigger (L1).

    Therefore, we must have that $i = j+1$ and is the dip $q$.

    Now suppose that the next inverse insertion step triggers (R3) instead of (R1/R2).
    Note that $a' = r_j > r_i = a$, so we must satisfy (I).
    Therefore, $s_i < a'$ and $a'=\min{[(s_i, r_{i+1})\cap A_i]}$.
    
    If $j < p$, then in order for inverse left insertion of $a$ to trigger (L1) we must have had that $a \neq \max \{r_{j+1}\}\cup ( A_j \cap (r_{j+1}, s_j))$.
    Note that $a \geq s_i$ as otherwise $\min{[(s_i, r_{i+1})\cap A_i]} \leq a$, a contradiction as $a' > a$.
    However, then $x$ witnesses $a$ over $s_i$.
    Thus, we must have that $j \geq p$.
    
    However, as left inverse inserting $a$ trigger (L1) we must have that (II)(b)(ii) is false meaning either $s_{j+1} \geq r_j$ or $s_{j+1}<r_j$ and there does not exist $k\geq j+2$ such that $s_k\in (s_{j+1}, r_j]$.
    
    If $s_{j+1} \geq r_j = a'$, then $\min{[(s_i, r_{i+1})\cap A_i]} = \emptyset$ as $s_i = s_{j+1} \geq r_j = r_{i+1}$.
    
    Otherwise, assume for each $k\geq j+2$ either $s_k < s_{j+1} = s_i$ or $s_k > r_j = a'$.
    Note that as every element in $R$ left of $r_j$ is greater than $a'$, we have that $A_j \cap (s_i, r_{i+1})$ can only contain elements greater than $a'$.
    Thus, $a' \neq \min{[(s_i, r_{i+1})\cap A_i]}$.
    This concludes the proof that inverse insertion preserves unimodality.\\
    
    It remains to show that after each inverse insertion step the tableau remains a (pseduo) strict decomposition tableau.
    By Lemma~\ref{l:SDT unimodal}, this is equivalent to saying if inverse insertion results in a tableau with one of the configurations either our original tableau had one of the configurations or we could forward insert ou element to get a tableau with one of the configurations.
    We will denote each case as a triple: the configuration found in $T'$, the element in the configuration that was inserted, and $\textbf{R}$ or $\textbf{L}$ depending on whether the insertion was right insertion or left insertion.
    For each case let $t$ be the element that was bumped out by the whatever created the bad configuration.
    Further, let $R = r_1 \dots r_\ell$ be the top of row of the configuration and $S = s_2 \dots s_k$ the bottom row.
    Note if $t$ was bumped by $u$ via right inverse insertion then $t < u$ and if via left insertion inverse $t > u$.\\

    \noindent\textbf{Case (i, a, L)}: Note that by definition of (L1) we either have that $a \geq r_{j+1}$ or $r_{j+1}$ is the dip where $t=r_j$.
    If we then right insert $a$ into $S$ by Lemma~\ref{lem:insertion_size} we have that $u>a$ where $u$ is the element $a$ bumps from right insertion.
    Therefore, after left inserting $v$ into $S$ we have $w > a$ where $w$ is the element below $r_{j+1}$, note $w\in S^{\downarrow}$.
    As either $a \geq r_{j+1}$ or $r_{j+1}$ is the dip we always have that $r_{j+1}$ over $w$ witnesses $a$.\\

    \noindent\textbf{Case (i, a, R)}: This requires $a\in R^{\downarrow}$ after inverse right insertion which cannot happen.\\

    \noindent\textbf{Case (i, b, L)}: As $t > b \geq a$ we have that $a$ and $t$ form configuration $(i)$.\\

    \noindent\textbf{Case (i, b, R)}: Left insert $b$ into $S$ will result in $b\in S$ meaning $a$ and $b$ form configuration $(i)$.\\

    \noindent\textbf{Case (ii, a, L)}: Note that by definition of (L1) we either have that $a \geq r_{j+1}$ or $r_{j+1}$ is the dip where $t=r_j$.
    If we then right insert $a$ into $S$ by Lemma~\ref{lem:insertion_size} we have that $u>a$ where $u$ is the element $a$ bumps from right insertion.
    Therefore, after left inserting $v$ into $S$ we have $w > a$ where $w$ is the element below $r_{j+1}$, note as $c\in S^{\downarrow}$ and is not the dip, we have that $w\in S^{\downarrow}$.
    As either $a \geq r_{j+1}$ or $r_{j+1}$ is the dip we always have that $r_{j+1}$ over $w$ witnesses $a$.\\

    \noindent\textbf{Case (ii, a, R)}: As $t < a$ we have that $t$, $b$ and $c$ form configuration $(ii)$.\\

    \noindent\textbf{Case (ii, b, L)}: Note that we must have $t\in S^{\downarrow}$ in order for $b$ to bump it when left inverse inserting.
    Therefore, we must have that $t < c$.
    However, as $t > b \leq a$ we then have that $a$, $t$, and $c$ form configuration $(ii)$.\\

    \noindent\textbf{Case (ii, b, R)}: Left inserting $b$ into $S$ will result in $b$ ending up right of $c$ as $c > b$ and $c\in S^{\downarrow}$.
    Therefore, after left insertion of $b$ into $S$ the tableau has configuration $(ii)$ made from $a$, $b$, and $c$.\\

    \noindent\textbf{Case (ii, c, L)}: As $t > c > b \geq a$ we have that $a$, $b$, and $t$ form configuration $(ii)$.\\

    \noindent\textbf{Case (ii, c, R)}: As $t\in S^{\uparrow}$ and is left of $b$ we have that $t < b$.
    However, as $c > b$ with $b$ right of $t$, $c$ cannot bump $t$ from inverse right insertion.
    Therefore, this case cannot happen.\\

    \noindent\textbf{Case (iii, v, L)}: Note that as $v < z$ we either have $v\in R^{\uparrow}$ or $v$ is the dip of $R$.
    However, by definition of (L1) we have that $v\in R^{\downarrow}$ and is not the dip of $R$.\\

    \noindent\textbf{Case (iii, v, R)}: As $t<v$ we have that $t$, $x$, $y$, and $z$ form configuration $(iii)$.\\

    \noindent\textbf{Case (iii, x, L)}: As  $t\in S^{\downarrow}$ and $z\in R^{\uparrow}$ with $z \geq y$ we have that $y$ witnesses $z$ over $t$.\\

    \noindent\textbf{Case (iii, v, R)}: As $t<x$ we have that $v$, $t$, $y$, and $z$ form configuration $(iii)$.\\

    \noindent\textbf{Case (iii, y, L)}: Note that as $x < y$ we have $y\in R^{\uparrow}$ meaning this case cannot happen.\\

    \noindent\textbf{Case (iii, y, R)}: If $x < t$, then as $t < y$ we have that $v$, $x$, $t$, and $z$ form configuration $(iii)$.
    Thus, suppose $x \geq t$ which requires that $x\in S^{\downarrow}$.
    However, as $z\in R^{\uparrow}$ we will have that $t$ witnesses $z$ over $x$.\\

    \noindent\textbf{Case (iii, z, L)}: Note that as $v < z$ we have $z\in R^{\uparrow}$ meaning this case cannot happen.\\

    \noindent\textbf{Case (iii, z, R)}: If $t \geq y$, then $v$, $x$, $y$, and $t$ form configuration $(iii)$, so suppose otherwise.
    Let $u$ be the element above $y$, as inverse right inserting $z$ bumps $t$ we have that $u > z$.
    However, after left inserting $z$ into $R$ we will then have that the tableau contains configuration $(iv/v)$ with $y$, $z$, and $u$.\\

    \noindent\textbf{Case (iv/v, x, L)}: If $t<y$, then $t$, $y$, and $z$ form configuration $(iv/v)$.
    Therefore, suppose $t\geq y$.
    However, then as $z\in R^{\uparrow}$ we have that $y$ witnesses $z$ over $t$.\\

    \noindent\textbf{Case (iv/v, x, R)}: As $t<x$ we have that $t$, $y$, and $z$ form configuration $(iv/v)$.\\

    \noindent\textbf{Case (iv/v, y, L)}: If $t\leq z$, then $x$, $t$, and $z$ form configuration $(iv/v)$.
    Therefore, suppose $t>z$.
    If we then right insert $y$ into $S$ we would have that $y$ ends up somewhere right of $x$ as $x < y$.
    Let $w$ be the element above where $y$ ends up.
    Note that as $z$ is above $x$ we have that $w$ is right of $z$ and therefore $z < w$.
    However, we then have that $z$ witnesses $w$ over $y$.\\

    \noindent\textbf{Case (iv/v, y, R)}: After left inserting $y$ into $R$ we have that $x$, $y$, and $z$ form configuration $(iv/v)$.\\

    \noindent\textbf{Case (iv/v, z, L)}: Note that as $y < z$ we have $z\in R^{\uparrow}$ meaning this case cannot happen.\\

    \noindent\textbf{Case (iv/v, z, R)}: If $y\in R^{\uparrow}$, then we require $t > y$ meaning $x$, $y$, and $t$ form configuration $(iv/v)$.
    Therefore, suppose $y\in R^{\downarrow}$.
    Let $u$ be the element $z$ bumps after being left inserted into $R$, note that $z > u \geq y > x$.
    As $u > x$ we have that after $u$ is right inserted into $S$, there is a $u$ right of $x$.
    Let $w$ be the element above that $u$, note that as $w$ is right of $t$ and when inverse right inserted $z$ bumps $t$ we have that $w > z$.
    Thus, $u$, $z$ and $w$ form configuration $(iv/v)$.
\end{proof}

\begin{definition}
    For a pair of a strict decomposition tableau and a standard shifted set valued Young tableau, $(T, S)$, 
    perform \emph{Kra\'{s}kiewicz--Hecke inverse insertion} by first finding the largest number, $n$, in $S$.
    If $n$ is alone in a box remove that box from $S$ and $T$ and left inverse insert the element that was in that box in $T$ to the row above.
    If $n$ is not alone in a box remove, $n$ from that box in $S$ and left inverse insert the element in that box in $T$ to the row above.
    If you would try to left inverse insert above the first row instead record the element you wished to left inverse insert as $a_n$.
    Repeat this process until your tableau are both empty, then $\mathbf{a} = (a_1, \ldots, a_n)$ is the $\emph{Kra\'{s}kiewicz--Hecke inverse}$ of $(T, S)$.
\end{definition}

Combining \Cref{t:invertible} with \Cref{p:row-insertion}, we have proved:

\begin{corollary}
\label{c:bijection}
	For all $n \in \mathbb{N}$, the map $\mathbb{N}^n \xrightarrow{\mathrm{KH}} \bigsqcup_{\lambda \vdash k \leq  n} \SDT(\lambda) \times \ShSet_n(\lambda)$ is a bijection.
\end{corollary}

\section{Conjectures}
\label{s:conjectures}

Our original aim in this project was to compute the $GQ_\lambda$--expansion of $G^C_w$.
As we explain in Section~\ref{ss:no-insertion}, insertion methods exhibit a fundamental inadequacy for the task.
Despite this shortcoming, empirically Kra\'skiewicz--Hecke insertion correctly computes this expansion as asserted in Conjecture~\ref{conj:expansion}, which we restate.
Recall $a_w^C(\lambda)$ is the number of strict decomposition tableaux of shape $\lambda$ whose reading word is a $0$--Hecke expression for $w$.

\begin{conjecture*}[Conjecture~\ref{conj:expansion}]
    For $w$ a signed permutation,
    \[
    G^C_w =\sum_{\lambda\ \mathrm{strict}} \beta^{|\lambda| -\ell(w)} a_w^C(\lambda) \cdot GQ^{(\beta)}_\lambda.
    \]
\end{conjecture*}

This conjecture has been tested to correctly compute $G^C_w$ up to degree $9$ with $w \in W_4$ (words of five letters) and up to degree $11$ with $w \in W_3$.

By taking the $x_1\dots x_n$ coefficient on each side, we see Conjecture~\ref{conj:expansion} and Corollary~\ref{c:hecke-enumeration} would imply Corollary~\ref{c:0-hecke-count}:
\[
|\HH_p(w)| = \sum_{\lambda\ \mathrm{strict}} a_w^C(\lambda) \cdot |\ShSet_p(\lambda)|.
\]
Note we have removed the factor $2^p$ from each side; on the left side by replacing $\UU_p(w)$ with $\HH_p(w)$ and on the right by not allowing barred entries in our tableaux.
We view this as very strong evidence for Conjecture~\ref{conj:expansion}.
To see why, assume the conjecture were to fail for $w \in W_n$.
By Proposition~\ref{p:GC-GQ}, we have
\[
G^C_w = \sum_{\lambda\ \mathrm{strict}} b_w^C(\lambda) \cdot GQ^{(\beta)}_\lambda
\]
for some coefficients $b_w^C(\lambda) \in \Z$, and conjecturally in $\N$.
Then for fixed $p$, we have
\begin{equation}
\label{eq:tableaux-relation}
\sum_{\lambda\ \mathrm{strict}} a_w^C(\lambda) \cdot |\ShSet_p(\lambda)|    = \sum_{\lambda\ \mathrm{strict}} b_w^C(\lambda) \cdot |\ShSet_p(\lambda)|.
\end{equation}
Additionally, when $|\lambda| = \ell(w)$, we know from~\cite{lam1995b} that $a_w^C(\lambda) = b_w^C(\lambda)$ for all $\lambda$.
Therefore, and especially assuming the positivity of $b_w^C(\lambda)$, the failure of Conjecture~\ref{conj:expansion} gives a highly non-trivial relation on the number of shifted set-valued standard tableaux of given sizes.

We now explain how Conjecture~\ref{conj:expansion} implies several other significant results.
The first is a combinatorial proof of~\cite[Conj. 5.14]{lewis2021enriched} and~\cite[Conj. 4.36]{marberg2022shifted}, which says that a skew $GQ$ function is a positive integer combination of $GQ$ functions:
\begin{conjecture}
\label{conj:skew}
    For any shifted shape $\lambda/\mu$, we have 
    \[
    GQ^{(\beta)}_{\lambda/\mu} = \sum_{\nu} a^C_{w(\lambda/\mu)}(\nu)\cdot GQ_\nu.
    \]
\end{conjecture}
Assuming Conjecture~\ref{conj:expansion}, this follows from Theorem~\ref{t:full-commutative}.
This basic fact about skew $GQ$'s has resisted explanation for several years.

When $\ell(\lambda) = k$ and $\mu = (k-1,\dots,1)$, note that $D_{\lambda/\mu}$ is the Young diagram for an ordinary partition $\nu$.
Recall $GS^{(\beta)}_\nu = GQ^{(\beta)}_{\lambda/\mu}$.
Since each $GS$ function is a skew $GQ$ function, they should be non-negative integer combinations of $GQ$ functions~\cite[Conj.~5.14]{lewis2021enriched}.
As a special case of Conjecture~\ref{conj:skew}, we would have the following combinatorial expansion for $GS$ functions:
\begin{conjecture}
    \label{conj:GS}
        Let $\nu$ be a (not necessarily strict) partition with $\ell(\nu) = k$.
    Define $\rho$ so that $\nu_i = \rho_i + i  - k$, and let $\mu = (k-1,k-2,\dots,1)$.
    Then
    \[
GS_\nu = \sum_{\lambda} a^C_{w(\rho/\mu)}(\lambda) \cdot GQ_{\lambda}.
    \]
\end{conjecture}

As mentioned in the introduction, it is an open problem to find the $GQ$ expansion for products of $GQ$ functions.
Conjecture~\ref{conj:expansion} implies a special case of this problem.
Recall for $a\leq b$ positive integers that $\tau(a,b) = (b+a-1,b+a-3,\dots,b-a+1)$.
Combining Conjecture~\ref{conj:expansion} with Corollary~\ref{c:GQ-trapzeoid}, we have:
\begin{conjecture}
    \label{conj:trapezoid}
    Let $a \leq b$ positive integers and $\lambda$ a strict shape with $\lambda_1 < k$.
    Then
    \[
GQ_\lambda \cdot GQ_{\tau(a,b)} = \sum_{\nu} a^C_{w(\lambda,k) \cdot w(a,b,k)}(\nu) \cdot GQ_\nu.
    \]
\end{conjecture}

\begin{example}
    For $w = \overline{521}34891067$, we have $a^C_w((10,3,2)) = 2$ with tableaux
    \begin{center}
    \begin{ytableau}
        9 & 6 & 4 & 3 & 2 & 1 & 0 & 6 & 7 & 8\\
        \none & 8 & 1 & 0\\
        \none & \none & 0 & 7
    \end{ytableau}\;,\;\;
    \begin{ytableau}
        9 & 6 & 4 & 3 & 2 & 1 & 0 & 6 & 7 & 8\\
        \none & 8 & 1 & 0\\
        \none & \none & 7 & 0
    \end{ytableau}\,.
    \end{center}
     This correctly computes the coefficient of $GQ_{10\,32}$ in $GQ_{521} \cdot GQ_{42}$.
\end{example}

The case where $a = 1$ is equivalent to computing the product $GQ_\lambda \cdot GQ_b$.
Such products are instances of the Pieri rule of Buch and Ravikumar~\cite{buch2012pieri}.
However, the version of this rule following from Conjecture~\ref{conj:trapezoid} is substantially different, as the Buch-Ravikumar rule is in terms of what they refer to as $KLG$ tableaux, which are skew set-valued semistandard tableaux with primed entries.
Using recurrences from~\cite{buch2012pieri}, the first author has verified Conjecture~\ref{conj:trapezoid} for the Pieri case~\cite{arroyoforthcoming}.

Trapezoids $\tau(a,b)$ are the unique strict partitions $\mu$ for which  $GQ_\mu$ is a $GS$ function.
As such, the product $GQ_\lambda \cdot GQ_{\tau(a,b)}$ takes the form $GQ$ times $GS$.
Combining Conjecture~\ref{conj:expansion} with Corollary~\ref{c:GQ-GS} would solve this more general problem:

\begin{conjecture}
    \label{conj:GQ-GS}
    Let $\lambda$ be a strict partition with $\lambda_1 < a$ and $\nu$ be a (not necessarily) strict partition with $\ell(\nu) = k$.
    Define $\rho$ so that $\nu_i = \rho_i + i - a - k$, and let $\mu = (a+k-1,a+k-2,\dots,a)$.
    Then
    \[
    GQ_\lambda \cdot GS_\nu = \sum_{\sigma} a^C_{w(\lambda)\cdot w(\rho/\mu)}(\sigma) GQ_\sigma.
    \]
\end{conjecture}

\section{Concluding Remarks}
\label{s:final}

\subsection{}

Fix a positive integer and let $I \subset [n-1] - \{1\}$ so that $i \in I$ implies $i-1,i+1 \notin I$.
Following~\cite[(4.9)]{lewis2021enriched}, the \defn{multipeak quasisymmetric function} of $I$ is
\[
K^{(\beta)}_I := \sum_{S} \beta^{|S| - n} x^S
\]
where the sum is over $n$--tuples $S = (S_1,\dots,S_n)$ of sets in $\Z \setminus \{0\}$ so that
\begin{itemize}
    \item $\max(S_i) \preceq \min(S_{i+1})$ for $i \in [n-1]$;
    \item $S_i \cap S_{i+1} \subseteq \Z_-$ if $i \in I$;
    \item $S_i \cap S_{i+1} \subseteq \Z_+$ if $i \notin I$.
\end{itemize}
Here $|S| = \sum_{i=1}^n |S_i|$ and $x^S = x^{S_1}\dots x^{S_n}$ where $x^A = x_{|a_1|} \dots x_{|a_k|}$ for $A = \{a_1, \dots,a_k\} \subseteq \Z$.

Let $\ShSet^*(\lambda)$ be the set of shifted set valued standard Young tableaux of shape $\lambda$ so that no consecutive numbers share the same box.
\begin{proposition}[{\cite[(4.14)]{lewis2021enriched}}]
For $\lambda/$ a shifted shape of size $n$, define
\[
GQ^{(\beta)}_{\lambda} := \sum_{T \in \ShSet^*(\lambda)} \beta^{|T| - n} \cdot K^{(\beta)}_{\Peak(T)}.
\]
\end{proposition}

By analogy, one might hope to identify a set of words $\HH^*(w) \subseteq \HH(w)$ so that
\[
G^C_w = \sum_{(a_1,\dots,a_p) \in \HH^*(w)} \beta^{p - \ell(w)} \cdot K^{(\beta)}_{\Peak((a_1,\dots,a_p))}.
\]
To identify such a set bijectively would require an insertion-like algorithm that preserves peak sets and records consecutive letters in the same box if and only if the word is in $\HH^*(w)$.
An obvious candidate for $\HH^*(w)$ is words without repeated consecutive entries.
Our insertion fails to prove this identity.
Even in Type A, none of the analogues of Hecke insertion in the literature have this property.
It would be interesting to find one.

\subsection{}
\label{ss:no-insertion}

To prove Conjecture~\ref{conj:expansion} using an insertion algorithm, one would want a map
\[
\binom{\bfi}{\bfa} \mapsto (P,Q)
\]
where $P$ and $Q$ are shifted tableau of the same shape with $P$ strict and $Q$ set-valued semistandard.
Naively, one would hope Kra\'skiewicz--Hecke insertion accomplishes this goal by replacing the value $j$ in $Q_{KH}(\bfa)$ with $i_j$.
In order for this operation to produce a set-valued semistandard tableau, we would need to show entries cells containing $\overline{i}$ are in distinct rows and those containing $i$ are distinct columns.
Equivalently, if $(a_k > \dots > a_\ell < \dots < a_m)$ is unimodal we require that $k,\dots,\ell$ are in distinct rows of $Q_KH(\bfa)$ and $\ell,\dots,m$ are in distinct columns.
Likewise, if $a_j <  a_{j+1} > a_{j+2}$, we see $j{+1}$ cannot be in the same column as $j$ or the same row as ${j+2}$.

We present an example that demonstrates a fundamental obstruction to producing an insertion algorithm that proves Conjecture~\ref{conj:expansion}.
For $u = s_1s_0s_2, \; v = s_1s_0s_2s_1, \; w = s_1s_0s_2s_1s_0$,
\[
G^C_u = \GQ_{(3)} + \GQ_{(21)}+\beta\GQ_{(31)}\,, \quad \, G^C_v = \GQ_{(31)}\,, \quad G^C_w = \GQ_{(32)}
\]
The strict decomposition tableaux witnessing these expansions are:
\[u:\ 
U^1 = \begin{ytableau}
1 & 0 & 2
\end{ytableau}\,, \quad
U^2 = \begin{ytableau}
2 & 0 \\
\none & 1
\end{ytableau}\,, \quad U^3 = \begin{ytableau}
2 & 0 & 2\\
\none & 1
\end{ytableau}\,\quad v:\ 
V = \begin{ytableau}
2 & 0 & 1 \\
\none & 1
\end{ytableau}\, \qquad w:\ 
W = \begin{ytableau}
2 & 1 & 0 \\
\none & 1 & 0
\end{ytableau}\,.
\]
For a word $\bfx = (x_1,\dots,x_k)$ and $j \leq k$, recall $\bfx[j] = (x_1,\dots,x_j)$.
Consider the words 
\[
\bfa = (1,0,2,0,1,0), \bfb = (1,2,0,2,1,0) \in \HH(w), \quad \text{so} \quad P(\bfa) = P(\bfb) = W
\]
in order for Proposition~\ref{p:reading-word} to hold.
Similarly, $P(\bfa[5]) = P(\bfb[5]) = V$ since $v$ is vexillary.
From Kra\'skiewicz insertion, we know $P(\bfb[3]) = U^2$.
We claim any reasonable generalization of Kra\'skiewicz insertion will have $P(\bfb[4]) = U^3$.
Since we are inserting $2$ into $U^2$, which must result in one of $U^2$ or $U^3$.
If we return $U^2$, we would have recording tableaux
\[
Q(\bfb[4]) = \begin{ytableau}
    1 & 2\\
    \none & \scriptstyle{3,4}
\end{ytableau}\,, \quad Q(\bfb[5]) = \begin{ytableau}
    1 & 2 & 5\\
    \none & \scriptstyle{3,4}
\end{ytableau}\,,
\]
violating the peak condition at values $3,4,5$.
Therefore, we must obtain $U^3$, but this would give us recording tableaux
\[
Q(\bfa) = \begin{ytableau}
    1 & 2 & \scriptstyle{3,5}\\
    \none & 4 & 6
\end{ytableau}\,, \quad Q(\bfb) = \begin{ytableau}
    1 & 2 & \scriptstyle{4,5}\\
    \none & 3 & 6
\end{ytableau}\,,
\]
However, then $\bfb$ violates the peak condition at position $3,4,5$.
Alternatively, we could have $P(\bfb[3]) = U^1$, but then $P(\bfa[3]) = U^2$, which violates the unimodality condition.

This example shows any insertion algorithm mapping words to strict decomposition tableaux that generalizes Kra\'skiewicz insertion and satisfies the reading word property in Proposition~\ref{p:reading-word} cannot prove Conjecture~\ref{conj:expansion} in the obvious fashion.
Moreover, even if we alter the definition of strict decomposition tableaux, the only reasonable alternative to $U^3$ is
\[
\begin{ytableau}
    1 & 0 & 2\\
    \none & 1
\end{ytableau}\,,
\]
where a similar analysis shows no such insertion can exist.
\begin{comment}
\begin{enumerate}
    \item The insertion is a bijection between the set of words for $w$ and a pair of shifted tableaux $P$ and $Q$.
    \item The insertion is defined step-by-step via a rule $U' = U \leftarrow x$, that is, depending only on the shifted tableau and the number inserted into it.
    \item Recording tableaux are standard shifted set-valued tableau.
    
    \item Valid insertion tableaux satisfy some of the properties of strict decomposition tableaux:
    \begin{itemize}
        \item Rows are strictly unimodal.
        \item The maximum numbers in each row of the insertion tableau strictly decrease.
        \item Reduced tableaux are valid insertion tableaux.
    \end{itemize}
    \item The insertion respects reading words as in Proposition~\ref{p:reading-word}: $\rho(U)\circ s_x = \rho(U')$.
    \item Consecutive unimodal subsequence corresponds to a \textbf{vee} shape in the recording tableau.
    \item A consecutive subsequence containing a true peak can not correspond to a vee shape in the recording tableau.
\end{enumerate}
\end{comment}

\subsection{}
In\cite{tamvakis2023tableau}, the author gives a tableau formula for $G^C_w$ for a family of signed permutations he calls skew.
To define these, a signed permutation $w$ is \defn{$k$--Grassmannian} if $w s_k < w$ and $ws_j > w$ for $j \neq k$.
A signed permutation $v$ is \defn{skew} if there exist $u,w$ both $k$--Grassmannian so that $w = v \cdot u$.
Note that the Grassmannian signed permutations we consider are precisely the 0--Grassmannian signed permutations.
While Tamvakis gives a tableau formula for $G^C_w$, his formula is in terms of unshifted tableaux.
We suspect his tableaux for the $0$--Grassmannian case are equivalent to standard set--valued tableaux, but this is not obvious.
Finding the $GQ$--expansion of $G^C_v$ for $v$ skew is an interesting open problem for $k \geq 1$ even in the cohomology ($\beta = 0)$ case.

\subsection{}
One consequence of Conjecture~\ref{conj:expansion} is the weaker statement:
\begin{conjecture}
\label{conj:positivity}
    For $w \in W_n$, we have $G^C_w \in \N[\beta][GQ^{(\beta)}_\lambda:\lambda\ \mathrm{strict}]$.
\end{conjecture}
This statement likely has a geometric proof.
Assuming this weaker conjecture, we suspect that~\eqref{eq:tableaux-relation} can be used to prove special cases of Conjecture~\ref{conj:expansion}.

As a heuristic for why this should be possible, we compare our situation to the expansion of Type A $K$--Stanley symmetric functions into the stable Grothendieck polynomials.
For $w \in S_\infty$, if the (Type A) Stanley symmetric function is $F_w = s_\lambda + s_\mu$, then the $K$--Stanley symmetric function has expansion
\[
G^A_w = G^{(\beta)}_\lambda + G^{(\beta)}_\mu + \beta G^{(\beta)}_\nu
\]
for some third shape $\nu$ (this is a consequence of transition for Grothendieck polynomials~\cite{lascoux2001transition}).
If the analogous statement were true for Type C $K$--Stanley symmetric functions, we would have for $w \in W_\infty$ that
\[
F_w^C = Q_\lambda + Q_\mu \quad \mbox{implies} \quad G^C_w = GQ^{(\beta)}_\lambda + GQ^{(\beta)}_\mu + \beta GQ^{(\beta)}_\nu
\]
for some shape $\nu$.
We claim without proof for $\lambda \neq \mu$ strict shapes that $|\ShSet_p(\lambda)| \neq |\ShSet_p(\mu)|$ for all $p$.
Therefore, assuming Conjecture~\ref{conj:positivity} we would necessarily have $a^C_w(\nu) = 1$ as desired.

\subsection{}
There is a parallel story of Type B $K$--Stanley symmetric functions $G^B_w$, also defined in~\cite{kirillov2017construction}.
These belong to the ring $\Z[\beta][GP_\lambda:\lambda \text{ strict}]$, where the $GP$'s are symmetric functions that represent $K$--theory of the orthogonal Grassmannian.
We limit our focus in this paper to the $GQ$ case as $GP$'s are better understood.
In particular, there already exist multiple combinatorial rules for multiplying $GP$ symmetric functions~\cite{clifford2014k,hamaker2017shifted}.
However, we remark that it is an open problem to express a skew $GP$ function in terms of $GP$'s~\cite[Conj. 4.35]{marberg2022shifted}.
It is easy to adapt the map $\res$ to compute $G^B_w$ for $w$ top.
Therefore, if one could prove a version of Conjecture~\ref{conj:expansion} fo the $G^B$ setting, this would solve this open problem.

\bibliographystyle{plain}
\bibliography{references}

\begin{thebibliography}{10}

\bibitem{anderson2019k}
David Anderson.
\newblock K-theoretic {C}hern class formulas for vexillary degeneracy loci.
\newblock {\em Advances in Mathematics}, 350:440--485, 2019.

\bibitem{arroyoforthcoming}
Joshua Arroyo.
\newblock An alternate {P}ieri rule for {$GQ$}’s.
\newblock {\em In Progress}, 2025+.

\bibitem{billey1995schubert}
Sara Billey and Mark Haiman.
\newblock {S}chubert polynomials for the classical groups.
\newblock {\em Journal of the American Mathematical Society}, pages 443--482, 1995.

\bibitem{billey1998vexillary}
Sara Billey and Tao~Kai Lam.
\newblock Vexillary elements in the hyperoctahedral group.
\newblock {\em Journal of Algebraic Combinatorics}, 8(2):139--152, 1998.

\bibitem{buch2002littlewood}
Anders~Skovsted Buch.
\newblock A {L}ittlewood-{R}ichardson rule for the {K}-theory of {G}rassmannians.
\newblock {\em Acta Mathematica}, 189(1):37--78, 2002.

\bibitem{buch2008stable}
Anders~Skovsted Buch, Andrew Kresch, Mark Shimozono, Harry Tamvakis, and Alexander Yong.
\newblock Stable {G}rothendieck polynomials and {K}-theoretic factor sequences.
\newblock {\em Mathematische annalen}, 340:359--382, 2008.

\bibitem{buch2012pieri}
Anders~Skovsted Buch and Vijay Ravikumar.
\newblock Pieri rules for the {K}-theory of cominuscule {G}rassmannians.
\newblock {\em Journal f{\"u}r die reine und angewandte Mathematik (Crelles Journal)}, 2012(668):109--132, 2012.

\bibitem{buch2016k}
Anders~Skovsted Buch and Matthew~J Samuel.
\newblock K-theory of minuscule varieties.
\newblock {\em Journal f{\"u}r die reine und angewandte Mathematik (Crelles Journal)}, 2016(719):133--171, 2016.

\bibitem{chiu2023expanding}
Yu-Cheng Chiu and Eric Marberg.
\newblock Expanding ${K}$-theoretic {S}chur ${Q}$-functions.
\newblock {\em Algebraic Combinatorics}, 6(6):1419--1445, 2023.

\bibitem{clifford2014k}
Edward Clifford, Hugh Thomas, and Alexander Yong.
\newblock K-theoretic {S}chubert calculus for {$OG(n, 2 n+ 1)$} and jeu de taquin for shifted increasing tableaux.
\newblock {\em Journal f{\"u}r die reine und angewandte Mathematik (Crelles Journal)}, 2014(690):51--63, 2014.

\bibitem{edelman1987balanced}
Paul {E}delman and Curtis {G}reene.
\newblock Balanced tableaux.
\newblock {\em Advances in Mathematics}, 63(1):42--99, 1987.

\bibitem{fomin1996combinatorial}
Sergey Fomin and Anatol Kirillov.
\newblock Combinatorial {$B_n$}-analogues of {S}chubert polynomials.
\newblock {\em Transactions of the American Mathematical Society}, 348(9):3591--3620, 1996.

\bibitem{fomin1994grothendieck}
Sergey Fomin and Anatol~N Kirillov.
\newblock Grothendieck polynomials and the {Y}ang-{B}axter equation.
\newblock In {\em Proc. 6th Conf. Formal Power Series and Alg. Comb}, pages 183--190, 1994.

\bibitem{graham2015excited}
William Graham and Victor Kreiman.
\newblock Excited {Y}oung diagrams, equivariant {K}-theory, and {S}chubert varieties.
\newblock {\em Transactions of the American Mathematical Society}, 367(9):6597--6645, 2015.

\bibitem{hamaker2017shifted}
Zachary Hamaker, Adam Keilthy, Rebecca Patrias, Lillian Webster, Yinuo Zhang, and Shuqi Zhou.
\newblock Shifted {H}ecke insertion and the {K}-theory of {$OG (n, 2n+ 1)$}.
\newblock {\em Journal of Combinatorial Theory, Series A}, 151:207--240, 2017.

\bibitem{hamaker2014bijective}
Zachary~R Hamaker.
\newblock {\em Bijective combinatorics of reduced decompositions}.
\newblock PhD thesis, Dartmouth College, 2014.

\bibitem{hawkes2024combinatorics}
Graham Hawkes.
\newblock Combinatorics of double grothendieck polynomials.
\newblock {\em Electronic Journal of Combinatorics}, P4.13, 2024.

\bibitem{hiller1986pieri}
Howard Hiller and Brian Boe.
\newblock Pieri formula for {$SO_{2n+1}/U_n$} and {$Sp_n/U_n$}.
\newblock {\em Advances in Mathematics}, 62(1):49--67, 1986.

\bibitem{ikeda2013k}
Takeshi Ikeda and Hiroshi Naruse.
\newblock {$K$}-theoretic analogues of factorial {S}chur {$P$}-and {$Q$}-functions.
\newblock {\em Advances in Mathematics}, 243:22--66, 2013.

\bibitem{kirillov2017construction}
Anatol~N Kirillov and Hiroshi Naruse.
\newblock Construction of double {G}rothendieck polynomials of classical types using id{C}oxeter algebras.
\newblock {\em Tokyo Journal of Mathematics}, 39(3):695--728, 2017.

\bibitem{knutson2013positroid}
Allen Knutson, Thomas Lam, and David~E Speyer.
\newblock Positroid varieties: juggling and geometry.
\newblock {\em Compositio Mathematica}, 149(10):1710--1752, 2013.

\bibitem{kraskiewcz1989reduced}
W~Kraskiewcz.
\newblock Reduced decompositions in hyperoctahedral groups.
\newblock {\em Comptes rendus de l'Acad{\'e}mie des sciences. S{\'e}rie 1, Math{\'e}matique}, 309(16):903--907, 1989.

\bibitem{lam1995b}
Tao~Kai Lam.
\newblock {\em B and {D} analogues of stable {S}chubert polynomials and related insertion algorithms}.
\newblock PhD thesis, Massachusetts Institute of Technology, 1995.

\bibitem{lascoux2001transition}
Alain Lascoux.
\newblock Transition on {G}rothendieck polynomials.
\newblock In {\em Physics and combinatorics}, pages 164--179. World Scientific, 2001.

\bibitem{lewis2021enriched}
Joel~Brewster Lewis and Eric Marberg.
\newblock Enriched set-valued {P}-partitions and shifted stable {G}rothendieck polynomials.
\newblock {\em Mathematische Zeitschrift}, 299(3-4):1929--1972, 2021.

\bibitem{manivel2001symmetric}
Laurent Manivel.
\newblock {\em Symmetric functions, {S}chubert polynomials and degeneracy loci}, volume~3.
\newblock American Mathematical Soc., 2001.

\bibitem{marberg2022shifted}
Eric Marberg.
\newblock Shifted combinatorial {Hopf} algebras from ${K}$-theory.
\newblock {\em Algebraic Combinatorics}, 7(4):1123--1156, 2024.

\bibitem{morse2020crystal}
Jennifer Morse, Jianping Pan, Wencin Poh, and Anne Schilling.
\newblock A crystal on decreasing factorizations in the 0-{H}ecke monoid.
\newblock {\em The Electronic Journal of Combinatorics}, 27(2):1--48, 2020.

\bibitem{pragacz2006algebro}
Piotr Pragacz.
\newblock Algebro—{G}eometric applications of {S}chur {S}-and {Q}-polynomials.
\newblock In {\em Topics in Invariant Theory: S{\'e}minaire d'Alg{\`e}bre P. Dubreil et M.-P. Malliavin 1989--1990 (40{\`e}me Ann{\'e}e)}, pages 130--191. Springer, 1991.

\bibitem{serrano2010shifted}
Luis Serrano.
\newblock The shifted plactic monoid.
\newblock {\em Mathematische Zeitschrift}, 266(2):363--392, 2010.

\bibitem{stanley1984number}
Richard~P {S}tanley.
\newblock On the number of reduced decompositions of elements of {C}oxeter groups.
\newblock {\em European Journal of Combinatorics}, 5(4):359--372, 1984.

\bibitem{stembridge1996fully}
John~R Stembridge.
\newblock On the fully commutative elements of {C}oxeter groups.
\newblock {\em Journal of Algebraic Combinatorics}, 5:353--385, 1996.

\bibitem{stembridge97}
John~R. Stembridge.
\newblock Some combinatorial aspects of reduced words in finite {C}oxeter groups.
\newblock {\em Trans. Amer. Math. Soc.}, 349(4):1285--1332, 1997.

\bibitem{tamvakis2023tableau}
Harry Tamvakis.
\newblock Tableau formulas for skew {G}rothendieck polynomials.
\newblock {\em Journal of the Mathematical Society of Japan}, 1(1):1--26, 2023.

\end{thebibliography}
\end{document}